\documentclass[12pt]{amsart}
\usepackage{dsfont,hyperref}
\usepackage[all]{xy}
\usepackage{amssymb}
\SelectTips{cm}{}
\usepackage{fullpage}
\allowdisplaybreaks

\newtheorem{theorem}{Theorem}[section]
\newtheorem{lemma}[theorem]{Lemma}
\newtheorem{proposition}[theorem]{Proposition}
\newtheorem{corollary}[theorem]{Corollary}
\theoremstyle{definition}
\newtheorem{definition}[theorem]{Definition}
\newtheorem{example}[theorem]{Example}
\newtheorem{question}[theorem]{Question}

\newtheorem{remark}[theorem]{Remark}

\newcommand{\Ext}{\text{Ext}}

\newcommand{\End}{\text{End}}
\newcommand{\Hom}{\text{Hom}}

\newcommand{\Rep}{\text{Rep}}

\newcommand{\C}{\mathcal{C}}

\newcommand{\mO}{{\mathcal O}}

\newcommand{\QQ}{\mathbb{Q}}
\newcommand{\ZZ}{{\mathbb{Z}}}

\newcommand{\one}{\mathds{1}}
\newcommand{\lk}{\ar@{-}}
\newcommand{\lkdl}{\lk[dl]}
\newcommand{\lkdr}{\lk[dr]}
\newcommand{\lkdlr}{\lkdl\lkdr}
\newcommand{\D}{\mathcal{D}}
\newcommand{\M}{\mathcal{M}}
\renewcommand{\Vec}{\mathsf{Vec}}
\newcommand{\sVec}{\mathsf{sVec}}
\newcommand{\Ver}{\mathsf{Ver}}
\renewcommand{\k}{\mathbf{k}}
\newcommand{\FPdim}{\mathop{\mathrm{FPdim}}\nolimits}
\newcommand{\ccup}{\text{\small $\bigcup$}}
\newcommand{\coplus}{\text{\small $\bigoplus$}}

\begin{document}

\title{Symmetric tensor categories in characteristic 2}

\author{Dave Benson and Pavel Etingof} 

\begin{abstract} 
We construct and study a nested sequence of finite symmetric tensor 
categories $\Vec=\C_0\subset \C_1\subset\cdots\subset \C_n\subset\cdots$
over a field of characteristic $2$ such that $\C_{2n}$ are
incompressible, i.e., do not admit tensor functors into 
tensor categories of smaller Frobenius--Perron dimension. 
This generalizes the category $\C_1$ described by Venkatesh \cite{Ven} and the category $\C_2$ defined by Ostrik. 
The Grothendieck rings of the categories $\C_{2n}$ and $\C_{2n+1}$ are
both isomorphic to the ring of real cyclotomic integers defined by a primitive $2^{n+2}$-th root of unity, 
$\mO_n=\ZZ[2\cos(\pi/2^{n+1})]$. 
\end{abstract} 

\maketitle

\section{Introduction}

Classical Pontrjagin duality reconstructs a (locally) compact 
abelian topological group from its character group, see for
example Morris~\cite{Morris:1977a}.  Tannaka--Krein
duality~\cite{Krein:1949a,Tannaka:1938a} generalizes this to an
arbitrary compact topological group $G$, at the expense of
increased complication. The theorem, in modern
terms, states that $G$ may be recovered from the tensor category
$\Rep(G)$ of representations together with the forgetful functor 
$F\colon\Rep(G)\to\Vec$ to
finite dimensional vector spaces. The group $G$ is reconstructed
as the group of natural transformations from the tensor functor
$F$ to itself, with a suitable topology. 

The theorem has been
further generalized and clarified in many 
directions, and has led to the notion of a Tannakian 
category~\cite{Cartier:1956a,Deligne/Milne:1982a,Joyal/Street:1991a,SaavedraRivano:1972a}.
This is by definition a rigid abelian tensor category over an
algebraically closed field 
with an exact, faithful linear tensor functor
$F$ to finite dimensional vector spaces. 
The group of natural transformations from $F$ to itself 
in this context is an affine group scheme $G$, and we obtain an equivalence
of our given category with $\Rep(G)$. In the case where $\C$ is 
\emph{finite}, i.e., equivalent to the category of representations of
a finite dimensional algebra, $G$ is a finite group scheme.

This naturally leads to the following question.
Given a symmetric tensor category $\C$, when does there exist such a functor $F$
to finite dimensional vector spaces?
It transpires that in order for $F$ to exist in some decent generality, 
there is another target category we need to consider, namely
the category $\sVec$ of finite dimensional super vector spaces. 
A super vector space is a $\ZZ/2$-graded vector space $V=V_0\oplus V_1$,
where $V_0$ and $V_1$ are the two graded pieces. For homogeneous elements $v\in V$,
we write $|v|$ for the parity of $v$, regarded as an element of $\ZZ/2$.  
The tensor product of two
super vector spaces is defined by 
\[ (V\otimes W)_0=V_0\otimes W_0\oplus V_1\otimes W_1;\qquad
(V\otimes W)_1=V_0\otimes W_1\oplus V_1\otimes W_0. \]
To make these into a symmetric tensor category, we need a tensor identity,
associativity isomorphisms, and
commutativity isomorphisms, in such a way that certain diagrams
commute
(see \cite{EGNO}, Definitions 4.1.1 and 8.1.12).
The critical one here is the commutativity isomorphism, which involves a sign.
The map $V \otimes W \cong W \otimes V$ is given on homogeneous elements 
by $v \otimes w \mapsto (-1)^{|v||w|}w\otimes v$.

If $\C$ is a symmetric tensor abelian category then a \emph{fiber functor} to an underlying 
category (such as vector spaces or super vector spaces) is an exact faithful
functor which takes tensor products to tensor products (up to natural isomorphism)
and the tensor identity to the tensor identity. 

In characteristic zero, 
Deligne~\cite{De2} proved that under certain mild assumptions
on size, every symmetric tensor category admits a fiber
functor to finite dimensional super vector spaces. The size assumption is that
given any object $X$, there exists a partition $\lambda$ such that the Schur functor
corresponding to $\lambda$ vanishes on $X$. This is satisfied whenever the
category has \emph{moderate growth}, meaning that objects have finite
composition lengths, and the composition lengths of tensor powers grow
at worst exponentially.

\begin{theorem}[Deligne]
Suppose that $\C$ is a symmetric tensor category of moderate growth over
an algebraically closed field of characteristic zero. Then there is a fiber functor $\C\to\sVec$,
unique up to natural isomorphism. 

Furthermore, there exists an affine super group scheme $G$ (i.e., a group scheme over the category
$\sVec$) with an element $z\in G$ satisfying $z^2=1$, and such that conjugation
by $z$ is the parity involution on $G$, which multiplies even elements of the
coordinate ring by $+1$ and odd elements by $-1$, and such that
the category $\C$ is 
equivalent to $\Rep(G,z)$, the category of representations of $G$ with 
$z$ acting as the parity involution.
\end{theorem}

We are now naturally led to the question of what happens in prime characteristic.
If the category $\C$ is finite and semisimple, then this is answered by a theorem of 
Ostrik~[O2]. Namely, he showed that under these hypotheses, $\C$ admits
a fiber functor to the category $\Ver_p(\k)$. This is the symmetric tensor abelian 
category obtained from $\Rep_\k(\ZZ/p)$ by the process of 
\emph{semisimplification}. This process quotients 
out the set of morphisms $f\colon V\to W$ with
the property that for all morphisms $g\colon W\to V$ 
the composite $f\circ g$ has trace zero. The result is a semisimple
symmetric tensor category with $p-1$ isomorphism classes 
of simple objects, corresponding to Jordan blocks of size between
one and $p-1$. For instance, if $p=2$ then $\Ver_p(\k)=\Vec_\k$, the category
of finite dimensional vector spaces. In this case, Ostrik's theorem implies
that $\C$ is equivalent
to representations of a finite semisimple group scheme $G$.
If $p=3$ then $\Ver_p(\k)=\sVec_\k$, the category of finite dimensional
super vector spaces. But for $p\geq 5$, the structure is more complicated.

One may wonder whether such a fiber functor 
exists without the semisimplicity assumption.
This turns out not to be the case. Namely, let $\C_1$ 
be the tensor category of representations of the Hopf 
algebra $\k[d]/d^2$ in characteristic two, 
with the commutativity isomorphism given by 
$$
s(v\otimes w)=w\otimes v+dw\otimes dv.
$$
Then $\C_1$ does not admit a fiber functor to $\Vec_\k$ (\cite{Ven}, 1.5). 

So one may ask whether, still in characteristic two, 
any category $\C$ as above has a fiber functor to $\C_1$. 
This is also false, as was recently shown by V.~Ostrik. 
Namely, Ostrik~\cite{O3} constructed a symmetric tensor category 
$\C_2$ equivalent to $\C_1\oplus \Vec_\k$ as an abelian category, 
which does not admit a fiber functor to $\C_1$.  
So one may ask whether any $\C$ as above has a fiber functor to $\C_2$ 
until the next counterexample is found, and so on. 

The goal of this paper is to construct and study an infinite 
ascending chain of finite symmetric tensor categories 
$$
\Vec_\k=\C_0\subset \C_1\subset \C_2\subset \C_3\subset\cdots
$$
in characteristic two such that $\C_{2n}$ is incompressible, i.e., does not admit
a tensor functor to a finite tensor category of smaller
Frobenius--Perron dimension (even a non-symmetric one). This shows that there is no {\it finite} symmetric  tensor category $\D$ over $\k$ such that any finite symmetric tensor category $\C$ over $\k$ admits a fiber functor to $\D$. 

This motivates the following question, which currently remains open. 

\begin{question} Consider the infinite category $\C_\infty=\ccup_{n\ge 0}\,\C_n$. Does any finite symmetric tensor category over $\k$ admit a 
fiber functor to $\C_\infty$?
\end{question} 

The organization of the paper is as follows. In Section 2 we state the
main theorem, giving a construction of the categories $\C_n$. In
Section 3 we prove the main theorem, and show that 
 $\C_{2n}$ categorifies the ring ${\mathcal O}_n=\ZZ[2\cos(\pi/2^{n+1})]$ (i.e., this ring is the Grothendieck ring of $\C_{2n}$). In Section 4 we study further properties of the categories $\C_n$, in particular show that they are incompressible, compute their Cartan matrices and prove their universal property. We also give an application of the categories $\C_n$ to modular representation theory of finite groups. In Section 5 we compute the structure of $\C_n$ for small $n$.\bigskip

\noindent
\textbf{Acknowledgements.} This material is based on work supported by
the National Science Foundation under Grant No.~DMS-1440140 while the
authors were in residence at the Mathematical Sciences Research
Institute in Berkeley, California, during the Spring 2018 semester.
The work of the second author was also partially supported by the NSF grant DMS-1502244.
The second author is very grateful to V. Ostrik for sharing his
construction of the category $\C_2$ and for numerous useful
discussions, without which this paper would not have appeared. 
The authors also thank A. Davydov, R. Guralnick, D. Nikshych, V. Serganova, P. Tiep and G. Williamson 
for useful discussions, and the referee for suggestions that helped improve the paper. 

\section{The main theorem} 

Throughout the paper, we will use the basics on tensor categories from \cite{EGNO}. 

Let $\k$ be an algebraically closed field of characteristic $p>0$. 
Consider the symmetric rigid monoidal category ${\rm Tilt}_p$ of
tilting modules over $SL(2,\k)$ (see \cite{J}, Appendix E). Its indecomposable objects are $T_m$,
$m\ge 0$ (the tilting module with highest weight $m$). It is known
that this category contains a descending chain of tensor ideals
$\mathcal{I}_1\supset \mathcal{I}_2\supset\cdots$, 
where $\mathcal{I}_n$ is spanned by the objects $T_m$, $m\ge p^n-1$. Let $\mathcal{T}_{n,p}={\rm Tilt}_p/\mathcal{I}_n$. If $n=1$, the category $\mathcal{T}_{n,p}$ is semisimple abelian and equivalent to the Verlinde category ${\rm Ver}_p={\rm Ver}_p(\k)$, but for $n\ge 2$ it is in general only Karoubian and not abelian. 

Now assume that $p=2$. Recall that in any symmetric tensor category $\C$ over $\k$ and $X\in \C$, 
the Frobenius twist $X^{(1)}$ is the cohomology (i.e., the kernel modulo the image) of the operator 
$$
d=1+s\colon X\otimes X\to X\otimes X,
$$ 
where $s$ is the commutativity isomorphism (note that $d^2=0$, so the cohomology is well defined). It is easy to show that the functor $F$ sending $X$ to $X^{(1)}$ is additive and "exact in the middle", i.e., for any short exact sequence 
$$
0\to X\to Y\to Z\to 0
$$
the sequence $F(X)\to F(Y)\to F(Z)$ is exact. In fact, an even stronger statement holds: 
we have a 3-periodic long exact sequence 
$$
\dots F(Z)\to F(X)\to F(Y)\to F(Z)\to F(X)\to\dots
$$
Indeed, the object $Y\otimes Y$ has a 3-step filtration with successive quotients $X\otimes X,X\otimes Z\oplus Z\otimes X$ and $Z\otimes Z$, so we have the corresponding 3-step filtration on the complex 
$$
\dots \to Y\otimes Y\to Y\otimes Y\to\dots, 
$$
with successive quotients having cohomology $F(X)$, $0$, and $F(Z)$ in every degree, respectively. 
Therefore, taking the 
2-step filtration of $Y\otimes Y$ with quotients $X\otimes X$ and $(Y\otimes Y)/(X\otimes X)$,
we obtain a short exact sequence of the corresponding complexes giving the claimed long exact sequence of cohomology. 

This implies that the composition series of $F(Y)$ is dominated by the composition series of $F({\rm gr}Y)=F(X\oplus Z)=F(X)\oplus F(Z)$. 

Let $K$ be the algebraic closure of the fraction field of the ring of Witt vectors $W(\k)$. 
Let ${\rm Ver}_{2^{n+1}}(K)$ denote the category of tilting modules over the quantum group $SL_q(2,K)$ at a $2^{n+2}$th primitive root of unity $q$ modulo negligible objects (see \cite{EGNO}, Subsection 8.18.2), and let ${\rm Ver}_{2^{n+1}}^+(K)$ be its even part, spanned by modules with even highest weight. 

Our main result is the following theorem. 

\begin{theorem}\label{maint} Let ${\rm char}(\k)=2$. 
There exists an ascending chain of finite symmetric tensor categories over $\k$, 
${\rm Vec}_\k=\C_0\subset \C_1\subset \C_2\subset\cdots$ 
(with fully faithful symmetric tensor embeddings) having the following properties. 

(i) For each $n\ge 0$, $\C_{2n}$ contains ${\mathcal T}_{n+1,2}$ as a generating full rigid monoidal Karoubian subcategory.  

(ii) There are no symmetric tensor functors $\C_{i+1}\to \C_i$ and no tensor functors $\C_{2i+2}\to \C_{2i+1}$, $i\ge 0$. 

(iii) There exists a (non-symmetric) tensor functor $F_i\colon \C_{2i+1}\to \C_{2i}$, $i\ge 0$. 
Namely, for $i\ge 1$, the category $\C_{2i}$ is equipped with a distinguished self-dual simple object $X_i$, and $\C_{2i+1}$ 
is the category of finite dimensional modules over the commutative
cocommutative Hopf algebra $A_i=\wedge X_i$ with a nontrivial
triangular structure defining the symmetric braiding of
$\C_{2i+1}$. For $i=0$, the construction is the same except that
$X_0=0$ and $A_0=\wedge \one$ with a nontrivial triangular structure. In particular, $\C_{2i+1}$ has the same 
simple objects as $\C_{2i}$.  

(iv) One has 
$$
\FPdim(\C_{2n})=\FPdim {\rm Ver}_{2^{n+1}}(K)=\frac{2^{n}}{\sin^2(\pi/2^{n+1})},
$$
$$
\FPdim(\C_{2n-1})=\FPdim {\rm Ver}_{2^{n+1}}^+(K)=\frac{2^{n-1}}{\sin^2(\pi/2^{n+1})}.
$$

(v) All simple objects of $\C_n$ are self-dual. 

(vi) One has $X_n^2=2\cdot \one+X_{n-1}$ in ${\rm Gr}(\C_{2n})$. Moreover, 
$$
{\rm FPdim}(X_n)=2\cos(\pi/2^{n+1})=\sqrt{2+\sqrt{2+\cdots+\sqrt{2+\sqrt{2}}}}\qquad (\text{$n$ roots}),
$$
and
$$
\dim(X_n)=0.
$$

(vii) One has $A_n=\one\oplus X_n\oplus \one$ as an object of $\C_{2n}$, 
with multiplication $X_n\otimes X_n\to \one$ being the evaluation morphism of $X_n$ composed with the 
(arbitrarily normalized) identification $X_n\cong X_n^*$ in the first component. In particular, 
$$
{\rm FPdim}(A_n)=2+2\cos(\pi/2^{n+1})=4\cos^2(\pi/2^{n+2}).
$$

(viii) As an abelian category, $\C_{2n+1}$ is indecomposable; moreover, all simple objects of this category occur in 
the projective cover of $\one$.  

(ix) We have $\C_{2n+2}=\C_{2n+1}\oplus \M_n$, and $\M_n$ is equivalent to $\C_{2n}$
as a $\C_{2n+1}$-bimodule category, where the bimodule structure is defined by the functor $F_n$ and the symmetric braiding of $\C_{2n+1}$. Under this equivalence, $X_{n+1}$ corresponds to $\one\in \C_{2n}$. In particular, $\C_{2n}$ has $2^n$ simple objects. Namely, the simple objects of $\C_{2n}$ 
are $X_S:={\text{\small$\bigotimes$}}_{i\in S}X_i$, where $S\subset\lbrace{1,\dots,n\rbrace}$. 
Thus, 
$$
\C_{2n}=\C_{2n-1}\oplus \C_{2n-3}\oplus\cdots\oplus \C_1\oplus \C_0
$$
as abelian categories, i.e., it has $n$ blocks. Also, we have $X_{n+1}\otimes X_{n+1}\cong A_n$, the regular $A_n$-module in 
$\C_{2n}$, as objects of $\C_{2n+1}$.  

(x) The projective cover of a simple object $Y\in \C_{2n+1}$ in $\C_{2n+1}$ (and $\C_{2n+2}$) 
has the form $A_n\otimes P_Y$, where $P_Y$ is the projective cover 
of $Y$ in $\C_{2n}$. 

(xi) One has $X_{n+1}^{(1)}=X_n$. 
\end{theorem} 

Theorem \ref{maint} is proved in the next section. 

The categories $\C_n$ for $n\le 2$ were known before. Namely, 
the category $\C_1$ is the category of modules over the Hopf algebra 
$\k[d]/d^2$, with braiding defined by 
$$
c=s (1\otimes 1+d\otimes d);
$$
it is a reduction modulo 2 of the category of supervector spaces $\sVec_K={\rm Ver}_4^+(K)$ and appears in \cite{Ven}. 
The category $\C_2$ is the category ${\mathcal T}_{2,2}$, which happens
 to be abelian; this fact was observed by V. Ostrik. It has two simple objects $\one$ and $X=X_1$ and one more 
 indecomposable $P$ -- the projective cover of $\one$, which is a nontrivial extension of $\one$ by $\one$, and one has 
 $X\otimes X=P$, so ${\rm FPdim}(X)=\sqrt{2}$. This category is a reduction modulo $2$ of the Ising category ${\rm Ver}_4(K)$. 

\begin{corollary}\label{grring} Let $\mO_n$ be the ring of integers in the field $\QQ(2\cos(\pi/2^{n+1}))=\QQ(\zeta_n+\zeta_n^{-1})$, 
where $\zeta_n=\exp(\pi i/2^{n+1})$; that is, $\mO_n=\ZZ[2\cos(\pi/2^{n+1})]$ (\cite{W}, p.16, Proposition 2.16). Then the map $\FPdim$ defines an isomorphism of the Grothendieck ring 
of $\C_{2n}$ onto $\mO_n$. Under this isomorphism, the simple objects $X_S$ of $\C_{2n}$ map to the numbers 
$d_S:=\prod_{j\in S}d_j$, where $d_j:=\zeta_n^{2^{n-j}}+\zeta_n^{-2^{n-j}}$.   
\end{corollary}

\begin{proof} It is well known that $[\QQ(\zeta_n+\zeta_n^{-1}):\QQ]=2^n$. 
Also, ${\rm FPdim}(X_S)=d_S$. It is clear that 
any number of the form $\zeta_n^r+\zeta_n^{-r}$, $r=1,\dots,2^n-1$ is a linear combination of $d_S$
(by taking the binary expansion of $r$). This implies that $\FPdim\colon {\rm Gr}(\C_{2n})\to \mathcal O_n$ is an isomorphism
mapping $X_S$ to $d_S$. 
\end{proof} 

\begin{corollary}\label{fusionrules} Let $S\subset [1,n]$ and $i\in [1,n]$. Let $k_S(i)$ be the largest integer $\le i$ which does not belong to $S$. Then in ${\rm Gr}(\C_{2n})$ we have 
$$
X_iX_S=X_{k_S(i)\cup S\setminus [k_S(i)+1,i]}+2\sum_{k=k_S(i)+1}^i X_{S\setminus [k,i]}
$$
where we agree that for any $S$, $X_{0\cup S}=0$. 
\end{corollary} 

\begin{proof} This follows immediately from Corollary \ref{grring}. 
\end{proof}

Corollary \ref{fusionrules}
 gives the following recursion to determine the structure constants $N_{ST}^U$ of the multiplication of ${\rm Gr}(\C_{2n})$, 
such that $X_SX_T=\sum_U N_{ST}^UX_U$ (where $S,T,U\in [1,n]$). Namely, given $S,T,U\subset [1,n-1]$, 
we have 
$$
N_{ST}^{n\cup U}=N_{n\cup S,T}^U=N_{S,n\cup T}^U=N_{n\cup S,n\cup T}^{n\cup U}=0,
$$
$$
N_{n\cup S,T}^{n\cup U}=N_{S,n\cup T}^{n\cup U}=N_{ST}^U,
$$
and
\begin{gather*}
N_{n\cup S,n\cup T}^U=N_{ST}^{[1+\max U,n-1]\cup U\setminus \max U}+2\sum_{k=\max U}^{n-1}N_{ST}^{[k+1,n-1]\cup U},\ U\ne \varnothing; \\
N_{n\cup S,n\cup T}^\varnothing=2\sum_{k=0}^{n-1}N_{ST}^{[k+1,n-1]}. 
\end{gather*}

It is easy to see that in fact this sum can contain at most one nonzero term. 
Namely, we have the following corollary. 

\begin{corollary}\label{pwr2} 
Let $K$ be the largest integer $\le n-1$ contained in none or both of the sets $S,T$. Then if $U\notin \varnothing$ then 
$$
N_{n\cup S,n\cup T}^U=\begin{cases}
N_{ST}^{[1+\max U,n-1] \cup U\setminus \max U},\ K=\max U\\
2N_{ST}^{[K+1,n-1]\cup U},\ K>\max U\\
2N_{ST}^{[\max U+1,n-1]\cup U}, \ K<\max U
\end{cases} 
$$
and 
$$
N_{n\cup S,n\cup T}^\varnothing=
2N_{ST}^{[K+1,n-1]}.
$$
In particular, all nonzero numbers $N_{ST}^U$ are powers of $2$. 
\end{corollary}

\begin{proposition}\label{Serre} 
The category $\C_{2r-1}$ is a Serre subcategory of $\C_{2r+1}$ (and therefore of each $\C_n$, $n\ge 2r-1$). 
\end{proposition} 

\begin{proof} Recall that $\C_{2r+1}$ is the category of $A_r$-modules in $\C_{2r}$. 
Suppose $Y$ is an extension of two objects of $\C_{2r-1}$ inside $\C_{2r+1}$; thus, $Y$ is an $A_r$-module in $\C_{2r}$. 
Then all the composition factors of $Y$ as an object of $\C_{2r}$ are of the form $X_S$, $S\subset [1,r-1]$. Since 
$\Hom(X_r\otimes X_S,X_T)$ vanishes for $S,T\subset [1,r-1]$, the action of $A_r$ on $Y$ must be trivial. Thus 
$Y\in \C_{2r}$. Since $\C_{2r-1}$ is a direct summand in $\C_{2r}$, we have $Y\in \C_{2r-1}$, as desired. 
\end{proof} 

\begin{corollary}\label{subcat}
The only tensor subcategories of $\C_n$ are $\C_m$, $m\le n$. 
\end{corollary}

\begin{proof} It suffices to assume that $n=2r+1$. Let $\D\subset \C_{2r+1}$ be a tensor subcategory. 
Let $s$ be the largest integer such that $s\in S$ and $X_S\in \D$. By Proposition \ref{Serre}, 
this implies that $\D\subset \C_{2s+1}$, so we may assume that $s=r$. Thus $S=T\cup r$, where 
$T\subset [1,r-1]$. Hence $X_S=X_T\otimes X_r$, so 
$X_S^2=X_T^2(2+X_{r-1})$ in ${\rm Gr}(\C_{2r})$. 
This implies that every composition factor of $X_T^2X_{r-1}$ belongs 
to $\D$, in particular, $X_{r-1}$. Applying Frobenius, we see that $X_i\in \D$ for any $i\in [1,r-1]$. 
Thus, $X_T\in \D$, hence every composition factor of $X_TX_S=X_T^2X_r$ is in $\D$, 
in particular $X_r$. Thus, $X_L\in \D$ for each $L\subset [1,r]$. But $X_{[1,r]}$ is a projective object which tensor generates 
$\C_{2r}$. Thus $\C_{2r}\subset \D$. 

Recall that $\C_{2r+1}$ is the category of $A_r$-modules in $\C_{2r}$. Suppose $\D\ne \C_{2r}$. Then 
there exists $Y\in \D$ which is a nontrivial $A_r$-module. We claim that $Y$ can be chosen a faithful $A_r$-module. 
Indeed, if $Y$ is not faithful, then the annihilator of $Y$ in $A_r$ is the unit object sitting in degree $2$, 
so $Y\otimes Y$ is faithful. Now, if $Y$ is faithful then $Y\otimes Y^*$ contains $A_r$ as a submodule (where 
the action of $A_r$ on $Y^*$ is taken to be trivial). Thus, the free $A_r$-module $A_r$ belongs to $\D$. Thus $A_r\otimes X\in \D$ for any $X\in \C_{2r}$, which implies that any $A_r$-module belongs to $\D$ (as it is a quotient of $A_r\otimes X$ for some $X$), i.e., $\D=\C_{2r+1}$. 
\end{proof} 

\section{Proof of the main theorem} 

We will construct the categories $\C_n$ inductively starting from $\C_0=\Vec_\k$. Assume they have been constructed until 
$\C_{2n}$, with the claimed properties (except (i)), and let us construct $\C_{2n+1}$, $\C_{2n+2}$. Then we will separately prove (i). 

\subsection{Construction of $\C_{2n+1}$ and proof of (iii),(x), and the second part of (iv)}

Consider the algebra 
$A_n=\one\oplus X_n\oplus \one$ as defined in the theorem.  
This is a commutative and cocommutative Hopf algebra (as we are in characteristic 2). 
Hence $\C_{2n+1}:=A_n$-mod is a finite tensor category. Moreover, since $A_n$ is a local algebra, the simple objects of 
$\C_{2n+1}$ are the same as those of $\C_{2n}$, and the projective cover of a simple object $Y\in \C_{2n}$ in $\C_{2n+1}$ is the free $A_n$-module $A_n\otimes P_Y$, where $P_Y$ is the projective cover of $Y$ in $\C_{2n}$. This implies that 
$$
\FPdim(\C_{2n+1})=\FPdim(\C_{2n})\FPdim(A_n)=
$$
$$
=\frac{2^n}{\sin^2(\pi/2^{n+1})}\cdot 
4\cos^2(\pi/2^{n+2})=\frac{2^n}{\sin^2(\pi/2^{n+2})}. 
$$
Now let us put a symmetric structure on $\C_{2n+1}$. We will define this structure by the formula 
$$
c:=s \circ R,
$$
where $R\in \Hom(\one, A_n\otimes A_n)$ is a triangular structure on $A_n$ (the universal $R$-matrix); then $R$ defines a functorial isomorphism $R_{YZ}\colon Y\otimes Z\to Y\otimes Z$ for any $Y,Z\in \C_{2n+1}$. 
Namely, let $\tau\colon \one \to X_n\otimes X_n$ be the coevaluation map, and 
define $R$ by the formula 
$$
R=1\otimes 1+\tau+a\otimes a,
$$
where $a\colon \one\to A_n$ is the tautological inclusion in the top degree. 
It is easy to see that under suitable normalizations of $\tau$ and $a$ we have
\begin{equation}\label{taua}
\Delta(a)=a\otimes 1+1\otimes a+\tau,\ \tau_{13}\cdot \tau_{23}=\tau\otimes a, 
\end{equation}
which implies that $R$ satisfies the hexagon axioms $(\Delta\otimes 1)(R)=R_{13}R_{23}$, 
$(1\otimes \Delta)(R)=R_{13}R_{12}$. Also, \eqref{taua} implies that 
$$
\tau^2=\dim(X_n)a\otimes a=0,
$$ 
so 
$R_{21}R=1\otimes 1$, hence $R$ is a triangular structure, which defines a symmetric braiding on $\C_{2n+1}$.  

Thus, we have established parts (iii),(x), and the second part of (iv). 

\subsection{Proof of (ii) for $\C_{2n+1}$} 
Let us show that there are no symmetric tensor functors $\C_{2n+1}\to \C_{2n}$. Observe that 
the Frobenius--Perron dimensions of the simple objects $X_S$ of $\C_{2n}$ are linearly independent over $\QQ$ 
(in fact, they form a $\QQ$-basis of $\QQ(\cos(\pi/2^{n+1}))$). This implies that any tensor functor $F\colon \C_{2n+1}\to \C_{2n}$ must map $X_S$ to $X_S$ for each $S$. In particular, $F(X_{1,\dots, n})=X_{1,\dots, n}$, which is projective, which implies that $F$ is a surjective tensor functor, in fact 
already so when restricted to $\C_{2n}$ (see \cite{EGNO}, Subsection 6.3). Therefore, by \cite{EGNO}, Proposition 6.3.4, $F|_{\C_{2n}}: \C_{2n}\to \C_{2n}$ is an equivalence, and we may assume (by post-composing with its inverse) that $F|_{\C_{2n}}: \C_{2n}\to \C_{2n}$ is the identity. 

Now let $F: \C_{2n+1}\to \C_{2n}$ be a symmetric tensor functor. Let $G=\underline{\rm Aut}_{\otimes}(F)$ be the group scheme in $\C_{2n}$ corresponding to $F$, 
and $K=\pi_1(\C_{2n})\subset G$ be the fundamental group of $\C_{2n}$ (\cite{De1}, Section 8). Since the functor $F$ is a split surjection (i.e., we have an inclusion $\iota: \C_{2n}\hookrightarrow \C_{2n+1}$ such that $F\circ \iota\cong {\rm Id}$), the inclusion $K\hookrightarrow G$ 
is a split injection. Thus, $G\cong K\ltimes N$, where $N$ is a closed normal group subscheme in $G$. Let $H=O(N)^*$ be the group algebra of $N$ (a cocommutative Hopf algebra in $\C_{2n}$). Then $\C_{2n+1}$ is the category of $H$-modules in $\C_{2n}$, and $F$ is the forgetful functor (forgetting the structure of an $H$-module). 

If $\C$ is a finite symmetric tensor category and $B$ is a Hopf algebra in $\C$ then $B-{\rm mod}$ is a finite tensor category and 
the regular objects of these categories (see \cite{EGNO}, Definition 6.1.6) satisfy the equality $R_{B-{\rm mod}}=B\otimes R_\C$
in $K_0(B-{\rm mod})$, hence ${\rm FPdim}(B-{\rm mod})={\rm FPdim}(B){\rm FPdim}(\C)$.  
Thus ${\rm FPdim}(H)=\FPdim(\C_{2n+1})/\FPdim(\C_{2n})=2+\FPdim(X_n)$. Hence the augmentation ideal $I$ of $H$ is isomorphic to $X_n\oplus \one$ (since there are no nontrivial extensions between $\one$ and $X_n$ in $\C_{2n}$), i.e., $H=\one\oplus X_n\oplus \one$. 
Moreover, $I$ acts trivially in simple $H$-modules (as they are just simple objects of $\C_{2n}$ with the trivial action of $H$), hence it is nilpotent (i.e., $H$ is local). This easily implies that 
$H\cong A_n$ as an algebra (namely, the product $X_n\otimes X_n\to \one\subset I$ is nonzero since $H$ is a Hopf algebra). Thus, we may assume that $F$ is isomorphic to $F_n$ as an additive functor. In other words, we may assume that $F$ is obtained from $F_n$ by a Drinfeld twist $J$ of the Hopf algebra $A_n$. The twist $J$ must have the form 
$$
J=1\otimes 1+\lambda \tau+\mu a\otimes a,
$$
where $\lambda,\mu\in \k$. Since the functor $F$ is symmetric, we must have $(J^{21})^{-1}RJ=1\otimes 1$, which implies 
$\tau=0$, a contradiction. This proves (ii) for $\C_{2n+1}$. 

\subsection{Proof of (viii)}
Let us show that every simple object of $\C_{2n+1}$  occurs in the projective cover $P_n(\one)$ of $\one$ in $\C_{2n+1}$; in particular, this category is indecomposable. Indeed, for $n=0$ this is clear, so we can consider $n\ge 1$. 
We know that every simple object of $\C_{2n-1}$ occurs in $P_{n-1}(\one)$, and $P_n(\one)=\wedge X_n\otimes 
P_{n-1}(\one)$. Since $\wedge X_n=\one\oplus X_n\oplus \one$, we conclude that 
$P_n(\one)$ involves all simple objects of $\C_{2n}$, hence of $\C_{2n+1}$. This proves (viii).  

\subsection{Construction of $\C_{2n+2}$}

Now let us construct the symmetric tensor category $\C_{2n+2}$. For this, we can use \cite{DR}, Theorem 2.8, 
for $\mathcal S=\C_{2n}$, $H=A_n$, $\sigma=1+a$, $\beta=1$, $g=1$, $\omega=1$, 
$\Gamma\colon H^*\to H$ defined by $R=(\sigma\otimes \sigma)\Delta(\sigma)^{-1}$, 
$\lambda(a)=1$, $\lambda(1)=0$. Namely, Theorem 2.8 of \cite{DR} 
constructs a braided category, and in our case it is easy to see that the braiding is symmetric (since $\sigma^2=1$).  

The tensor category $\C_{2n+2}$ may also be constructed as a $\ZZ/2$-extension of $\C_{2n+1}$ using the methods of 
\cite{ENO} extended to the non-semisimple case in \cite{DN1}. Namely, we define 
$\C_{2n+2}=\C_{2n+1}\oplus \M_n$, where $\M_n=\C_{2n}$ as a $\C_{2n+1}$-module category
(which is exact since the functor $F_n$ is surjective), and make it a $\C_{2n+1}$-bimodule category 
by using the symmetric braiding on $\C_{2n+1}$. To define the tensor structure on $\C_{2n+2}$, we first need 
to show that the bimodule category $\M_n$ is invertible, and moreover defines an involution in the Brauer-Picard group 
${\rm BrPic}(\C_{2n+1})$. To this end, note that the $R$-matrix of $A_n$ is nondegenerate, i.e., defines a Hopf algebra isomorphism $\Gamma\colon A_n^*\to A_n$, and this isomorphism is symmetric since $R_{21}=R$. Hence, analogously to \cite{ENO}, Proposition 9.3, $\M_n$ is an involution in ${\rm BrPic}(\C_{2n+1})$, as desired. 

Now, according to \cite{ENO}, the first obstruction to lifting this structure to a structure of a tensor category on $\C_{2n+2}$ 
lies in $H^3(\ZZ/2,I)$, where $I$ is the group of invertible objects of the Drinfeld center of $\C_{2n+1}$. This Drinfeld center is the category of pairs $(Z,\phi)$, where $Z\in \C_{2n+1}$ and $\phi\colon Z\otimes ?\to ?\otimes Z$ is a functorial isomorphism satisfying consistency conditions. If $(Z,\phi)$ is invertible then $Z$ is invertible, hence $Z=\one$. 
Thus $I={\rm Aut}_\otimes({\rm Id})$, the category of tensor automorphisms of the identity functor of $\C_{2n+1}$.

We claim that $I=1$. Indeed, since $\C_{2n+1}$ is indecomposable, any $g\in I$ must act trivially on every simple object of $\C_{2n+1}$, in particular, on $X_{1,\dots,n}$, which is projective in $\C_{2n}$. Hence $g$ acts trivially on any object $Y$ of $\C_{2n}$, as $Y$ is a quotient of its projective cover $P_Y$, which is a direct summand 
in $Z\otimes X_{1,\dots,n}$ for some $Z\in \C_{2n}$. Thus, $g$ correspond to a grouplike element of the Hopf algebra $A_n$ (i.e., a character of $A_n^*$). But $A_n^*\cong A_n$, and its unique character is the counit, as desired.  
 
Hence the first obstruction vanishes. 
Also $H^2(\ZZ/2,I)$ vanishes, hence there is a unique way to pass to the next step 
of the extension process. 

The second obstruction now lies in $H^4(\ZZ/2,\k^\times)$, which vanishes. So 
the extension exists, and moreover it is unique since $H^3(\ZZ/2,\k^\times)$ also vanishes
(as we are in characteristic $2$). 

Moreover, by using a version of the extension theory of \cite{ENO} for braided and symmetric categories (\cite{DN2}) the category $\C_{2n+1}$ can be endowed with a braided  structure, which in fact turns out to be symmetric. Namely, following \cite{DN2}, to a finite braided tensor category $\C$ over $\k$ one can attach a simply connected classifying space $B_2\C$ with three nontrivial homotopy groups, $\pi_2$, $\pi_3$ and $\pi_4$. Then if $G$ is an abelian group then $G$-extensions of $\C$ as a braided category correspond to homotopy classes of maps $\phi: K(G,2)\to B_2\C$, where $K(G,2)$ is the Eilenberg-Mac Lane space attached to $G$. Such maps can also be classified using classical obstruction theory. Namely, at the first step we have to fix a homomorphism $\rho: G\to \pi_2$. Then we get the first obstruction $O_4(\rho)\in H^4(K(G,2),\pi_3)$. 
If it vanishes, then we go to the second step and get to make a choice $b$ in a torsor over $H^3(K(G,2),\pi_3)$, and we get the second obstruction $O_5(\rho,b)\in H^5(K(G,2),\pi_4)$. If it vanishes, then we go to the third step and get to choose an element $\gamma$ in a torsor over $H^4(K(G,2),\pi_4)$, and it determines the map $\phi$ uniquely up to homotopy. 

In our situation, $\C=\C_{2n+1}$, $G=\Bbb Z/2$, and we have already chosen the map $\rho$ given by the 
module category $\M_n$. Also, $\pi_3$ is the group of invertible objects in the symmetric center of $\C_{2n+1}$, i.e., $\pi_3=1$. Thus, $O_4$ automatically vanishes and we have a unique choice for $b$. Finally, $\pi_4=\k^\times$, 
so since $H_4(K(G,2))$ and $H_5(K(G,2))$ are 2-groups (see e.g. \cite{EM}, Theorem 22.1), we get 
$H^4(K(G,2),\pi_4)=H^5(K(G,2),\pi_4)=0$ (as ${\rm char}(\k)=2$, so $\k^\times$ is a uniquely 2-divisible group).
This implies that the map $\phi$ (and hence the corresponding braided extension) is uniquely determined by $\rho$. 

Furthermore, by using the version of the same theory for symmetric categories (with $H^j$ replaced by $H^{j+1}$, $\pi_i$ with $\pi_{i+1}$, $B_2$ with $B_3$ and $K(G,2)$ replaced by $K(G,3)$), also described in \cite{DN2}, one can show that the obtained braided extension is in fact symmetric. This can also be seen using Remark \ref{allsym} below. 

\subsection{Proof of (ii),(iv),(v),(vi) and (ix)} 
Now we have a symmetric tensor category $\C_{2n+2}$. By \cite{EGNO}, Theorem 3.5.2\footnote{Theorem 3.5.2 of \cite{EGNO} applies to semisimple categories, but the proof extends in a straightforward way to the non-semisimple case.}, we have $\FPdim(\C_{2n+1})=\FPdim(\M_n)$, hence $\FPdim(\C_{2n+2})=2\FPdim(\C_{2n+1})$, which yields (iv).  
Also, $\C_{2n+2}$ contains a simple object $X_{n+1}\in \M_n$ corresponding to $\one\in \C_{2n+1}$, and the simple objects of $\C_{2n+2}$ are of the form $X_S$ and $X_S\otimes X_{n+1}$, where $S\subset \lbrace{ 1,\dots,n\rbrace}$. 
This establishes (v) (since $X_{n+1}$ is the only object of smallest Frobenius-Perron dimension in $\M_n$) and also (ix), except its final statement. Also 
for any $Y\in \C_{2n+1}$ we have 
$$
\Hom_{\C_{2n+1}}(X_{n+1}\otimes X_{n+1},Y)=\Hom_{C_{2n+2}}(X_{n+1},Y\otimes X_{n+1})=\Hom_{\C_{2n}}(\bold 1,F(Y))=
\Hom_{\C_{2n+1}}(A_{n},Y),
$$ 
where $F: \C_{2n+1}\to \C_{2n}$ is the forgetful functor. This proves the final statement of (ix). 

Now, since 
$\FPdim(\C_{2n+1})=\FPdim(\M_n)$,   
we have 
$$
\FPdim(X_{n+1})^2=\frac{\FPdim(\C_{2n+1})}{\FPdim(\C_{2n})}=4\cos^2(\pi/2^{n+2}),
$$
hence $\FPdim(X_{n+1})=2\cos(\pi/2^{n+2})$. Also for Frobenius--Perron dimension reasons, in the Grothendieck group ${\rm Gr}(\C_{2n+2})$, we have 
$$
X_{n+1}^2=2\cdot \one+X_n,
$$ 
hence $\dim(X_{n+1})=0$, establishing (vi). Alternatively, these statements follow from (ix). 
Finally, once again for dimension reasons, there are no tensor functors $\C_{2n+2}\to \C_{2n+1}$, establishing (ii). 

\subsection{Proof of (vii) and (xi)}

We will need the following lemmas.

\begin{lemma}\label{frob} Let $\C$ be a symmetric tensor category over a field of characteristic $2$, and $X\in \C$. 
Then $X\otimes X$ has a 3-step filtration whose successive quotients are $\wedge^2X,X^{(1)},\wedge^2X$. 
\end{lemma} 

\begin{proof} The filtration is given by $F_1={\rm Im}(1+s)$, $F_2={\rm Ker}(1+s)$, $F_3=X\otimes X$. 
\end{proof} 

\begin{lemma}\label{wedge2} Let $\C$ be a symmetric tensor category over a field of any characteristic, and $X\ne 0$ be an object of $\C$ such that $\wedge^2X=0$ (or, equivalently, $S^2X\cong X\otimes X$). Then $X$ is invertible. 
\end{lemma} 

\begin{proof} Let us first show that $X$ is simple. Assume the contrary, i.e., that we have a 2-step filtration on $X$ with successive quotients $Y,Z\ne 0$. Then ${\rm gr}S^2X$ is a quotient of 
$$
S^2{\rm gr}X=S^2(Y\oplus Z)=S^2Y\oplus (Y\otimes Z)\oplus S^2Z.
$$ 
This is a contradiction with $S^2X\cong X\otimes X$, since 
$$
{\rm gr}(X\otimes X)=Y\otimes Y\oplus 2(Y\otimes Z)\oplus Z\otimes Z,
$$ 
i.e., has strictly larger length than ${\rm gr}S^2X$. 

Now note that $s_{X,X}=1$, hence $s_{X\otimes X^*,X\otimes X^*}=1$.
Thus by the above argument $X\otimes X^*$ is simple, hence $X\otimes X^*=\one$ and $X$ is invertible. 
\end{proof}

Now we are ready to prove (vii) and (xi). Namely, let us compute $\wedge^2X_{n+1}$ and $\wedge^3X_{n+1}$. 
By Lemma \ref{frob} we have a 3-step filtration on $X_{n+1}\otimes X_{n+1}$ with successive quotients 
$\wedge^2X_{n+1}$, $X_{n+1}^{(1)}$, $\wedge^2X_{n+1}$. Since $X_{n+1}$ is not invertible, by Lemma \ref{wedge2} this implies that $\wedge^2X_{n+1}=\one$  and $X_{n+1}^{(1)}=X_n$, proving (xi). 

Also note that the cyclic permutation $c$ is not the identity on $X_{n+1}^{\otimes 3}$ (as by Lemma \ref{wedge2}, $s_{X_{n+1},X_{n+1}}\ne 1$). This implies that $1+c+c^2$ is not invertible (as $c^3-1=0$). Thus ${\rm Ker}(1+c+c^2)=V\otimes \k^2$, where $\k^2$ is the 2-dimensional irreducible $S_3$-module and $V$ a nonzero object (namely, $V=S_{(2,1)}X_{n+1}$, the Schur functor attached to the partition $(2,1)$, which is well defined in characteristic $2$). But $X_{n+1}^{\otimes 3}=2X_{n+1}\oplus X_n\otimes X_{n+1}$, hence $V=X_{n+1}$. But we also have $X_{n+1}=
\wedge^2X_{n+1}\otimes X_{n+1}=\wedge^3X_{n+1}\oplus V$. Hence $\wedge^3X_{n+1}=0$, proving (vii). 

\subsection{The invariants in tensor powers of $X_n$}

Let us compute the dimension of the space of invariants in $X_n^{\otimes r}$. 
It is clear that if $r$ is odd then this space is zero, so it suffices to compute
the invariants in $X_n^{\otimes 2m}$. Let $d_{mn}=\dim \Hom(\one,X_n^{\otimes 2m})$. 

\begin{proposition}\label{inva} 
Let $f_n(z)=\sum_{m\ge 0}d_{mn}z^m$.
Then 
\begin{equation}\label{fnz}
f_n(z)=\frac{(t+t^{-1})(t^{2^{n+1}-1}-t^{-2^{n+1}+1})}{t^{2^{n+1}}-t^{-2^{n+1}}},
\end{equation}
where $z=(t+t^{-1})^{-2}$. 
\end{proposition} 

\begin{proof}
Since $X_0=0$, we have $f_0(z)=1$. 

Let $n,m\ge 1$. Recall that by Theorem \ref{maint}(ix) we have $X_n\otimes X_n=A_{n-1}$, the regular representation of the algebra $A_{n-1}$ as an object of $\C_{2n-1}$. Thus, 
$$
d_{mn}=\dim \Hom_{\C_{2n-1}}(\one,A_{n-1}^{\otimes m})=\dim \Hom_{\C_{2n-1}}(A_{n-1},A_{n-1}^{\otimes m-1}),
$$
using that $A_{n-1}\cong A_{n-1}^*$ as an $A_{n-1}$-module. Thus, by Frobenius reciprocity, we get
$$
d_{mn}=\dim \Hom_{\C_{2n-2}}(\one,A_{n-1}^{\otimes m-1})=
\dim \Hom_{\C_{2n-2}}(\one,(2\cdot \one\oplus X_{n-1})^{\otimes m-1}).
$$
This implies that 
$$
d_{mn}=\sum_{s\ge 0}\binom{m-1}{2s}2^{m-1-2s}d_{s,n-1}.
$$
In terms of generating functions, this recursion has the form 
$$
f_n(z)=1+\frac{z}{1-2z}f_{n-1}\left(\frac{z^2}{(1-2z)^2}\right),
$$
and it determines $f_n(z)$ for all $n$ from the initial condition $f_0(z)=1$. 

It remains to observe that function \eqref{fnz} satisfies this recursion. 
\end{proof} 

\subsection{Proof of (i)}
It remains to prove (i). For this we will use the following universal property of the tilting category. 

\begin{proposition}\label{uni} 
Let $\D$ be a symmetric tensor category over a field $\k$ of any characteristic. Then $\k$-linear symmetric monoidal functors $F\colon {\rm Tilt}(SL(2,\k))\to \D$ correspond to objects 
$X\in \D$ such that 

(1) $\wedge^2X\cong \one$ and $\wedge^3X=0$, or 

(2) $X=0$ in characteristic $2$, or 

(3) $X$ is an odd invertible object in characteristic $3$ (where ``odd" means that the braiding acts on $X\otimes X$ by $-1$). 

The correspondence is given by $F\mapsto F(T_1)$.
\end{proposition} 

\begin{proof} Analogous to the proof of Theorem 2.4 of \cite{O1}.\footnote{See \cite{BEO} for more details.} 
\end{proof} 

Let $Q_{n,p}$ be the polynomial (with integer coefficients) defined by the formula
$$
Q_{n,p}(2\cos x)\colon =\frac{\sin (p^nx)}{\sin x}.
$$ 
We write $Q_{n,p}=Q_{n,p}^+-Q_{n,p}^-$, where $Q_{n,p}^+$ is the sum of all the terms of $Q_{n,p}$ 
with positive coefficients and $Q_{n,p}^-$ is minus the sum of the terms with negative coefficients. 

The following corollary will be used in the proof of Theorem \ref{unipro}.

\begin{corollary}\label{uni1}
Let $\D$ be a symmetric tensor category over $\k$. Then $\k$-linear symmetric monoidal functors $F\colon \mathcal{T}_{n,p}\to \D$ correspond to objects 
$X\in \D$ such that $Q_{n,p}^+(X)\cong Q_{n,p}^-(X)$ and 

(1) $\wedge^2X\cong \one$ and $\wedge^3X=0$, or 

(2) $X=0$ in characteristic $2$ ($n=1$), or 

(3) $X$ is an odd invertible object in characteristic $3$ ($n=1$). 

The correspondence is given by $F\mapsto F(T_1)$.
\end{corollary}

\begin{proof} The tensor ideal $\mathcal I_n$ is generated by the object 
$T_{p^n-1}$. This is the irreducible standard module ${\rm St}_n$ over $SL(2,\k)$ with highest weight 
$p^n-1$ (the $n$-th Steinberg representation). Since characters of tilting modules are linearly independent, 
in the split Grothendieck ring of ${\rm Tilt}(SL(2,\k))$ we have $T_{p^n-1}=Q_{n,p}(T_1)$ (since this identity holds at the level of characters, as $T_{p^n-1}$ is a standard module). Hence $T_{p^n-1}\oplus Q_{n,p}^-(T_1)\cong Q_{n,p}^+(T_1)$. This implies that 
if $X=F(T_1)$ then $Q_{n,p}^-(X)$ is always a direct summand in $Q_{n,p}^+(X)$, and 
for a monoidal functor $F: {\rm Tilt}(SL(2,\k))\to \D$, the condition that $F(Q_{n,p}^+(T_1))\cong F(Q_{n,p}^-(T_1))$ (i.e., 
$Q_{n,p}^+(X)\cong Q_{n,p}^-(X)$) is equivalent to the condition that $F(T_{p^n-1})=0$. 
Thus, the corollary follows from Proposition \ref{uni}. 
\end{proof} 

Now let us return to the case ${\rm char}(\k)=2$. To simplify notation, we write $Q_n$ and $Q_n^\pm$ for $Q_{n,2}$ and $Q_{n,2}^\pm$. In $\C_{2n}$, $n\ge 1$,  we have the object $X_n$ such that $\wedge^2X_n=\one$, $\wedge^3X_n=0$. Recall also that $X_0=0$. 
Hence by Proposition \ref{uni} we have a symmetric monoidal functor 
$$
F\colon {\rm Tilt}(SL(2,\k))\to \C_{2n}.
$$ 
which sends $T_1$ to $X_n$. Moreover, we have 
$$
F(T_{2^{n+1}-1})=Q_{n+1}(X_n)
$$
in the Grothendieck ring of $\C_{2n}$. But $Q_{n+1}(X_n)=0$, since the eigenvalues of multiplication by $X_n$ are Galois conjugates of $2\cos(\pi/2^{n+1})$, which are exactly the roots of $Q_{n+1}$. This implies that the class of $F(T_{2^{n+1}-1})$ in the Grothendieck ring of $\C_{2n}$ is zero, 
thus $F(T_{2^{n+1}-1})$ is itself zero.

Thus, we obtain 

\begin{corollary}\label{symmon}
We have a symmetric monoidal functor $F\colon {\mathcal T}_{n+1,2}\to \C_{2n}$ such that $F(T_1)=X_n$. 
\end{corollary} 

It remains to show that the functor $F$ of Corollary \ref{symmon} is fully faithful. To this end, we first prove the following lemma. 

\begin{lemma}\label{faith} 
Let $\C$ be a Karoubian monoidal category over a field $\k$ of characteristic $p$ and $F\colon {\mathcal T}_{n+1,p}\to \C$ be an additive monoidal functor. Let ${\rm St}_n=T_{p^n-1}$ be the $n$-th Steinberg module. If $F({\rm St}_n)\ne 0$ then $F$ is faithful. 
\end{lemma}

\begin{proof}
 It suffices to show that if $f\colon \bold 1\to Z$ is a morphism in ${\rm Tilt}_{n+1,p}$ 
such that $F(f)=0$ then $f=0$. Assume the contrary, i.e., that $f\ne 0$. Let ${\mathcal I}(f)$ be the tensor ideal generated by $f$. Consider the morphism $f\otimes 1\colon {\rm St}_n\to Z\otimes {\rm St}_n$. We have $Z\otimes {\rm St}_n={\rm St}_n^{\oplus m}\oplus M$, where $M$ is a direct sum of $T_i$, $p^n\le i\le p^{n+1}-1$. Moreover, by the linkage principle (\cite{J}, Section 6) we have 
$\Hom({\rm St}_n,M)=0$. Hence there exists $g\colon Z\otimes {\rm St}_n\to {\rm St}_n$ such that $g\circ (f\otimes 1)={\rm Id}_{{\rm St}_n}$, i.e., ${\rm Id}_{{\rm St}_n}\in \mathcal{I}(f)$. Thus, since the functor $F$ kills ${\mathcal I}(f)$, it kills ${\rm Id}_{{\rm St}_n}$, hence kills ${\rm St}_n$ itself, a contradiction. 
\end{proof} 

\begin{proposition}\label{faith1} The functor $F$ is faithful. 
\end{proposition} 

\begin{proof} By Lemma \ref{faith}, it suffices to show that $F({\rm St}_n)\ne 0$. 
But $F({\rm St}_n)=Q_n(T_1)$, and any polynomial 
$Q$ of degree less than $2^n$ such that $Q(X_n)=0$ is identically zero (as $X_n$ has $2^n$ distinct eigenvalues). 
This implies the statement. 
\end{proof}

\begin{proposition}\label{fullf}
The functor $F$ is full. 
\end{proposition} 

\begin{proof} It suffices to check that $F$ is full on objects of the form 
$T_1^{\otimes r}$. Thus, since $F$ is faithful, it suffices to check 
that 
$$
\dim \Hom(\one,T_1^{\otimes 2m})=d_{mn}. 
$$

Let $D_{mn}:=\dim \Hom(\one,T_1^{\otimes 2m})$. 
Let $\widetilde T_i$ be the lift of $T_i$ to the category ${\rm Ver}_{2^n}(K)$. 
Then, since $\Ext^1_{SL(2,\k)}(T_i,T_j)=0$, we have 
$\dim \Hom(T_i,T_j)=\dim\Hom(\widetilde T_i,\widetilde T_j)$. 
Hence
$$
D_{mn}=\dim \Hom(\one,\widetilde T_1^{\otimes 2m}).
$$
But $\dim \Hom(\one,\widetilde T_1^{\otimes 2m})$ just equals the number of paths 
on the Dynkin diagram of type $A_{2^{n+1}-1}$ of length $2m$ beginning and ending at the left end.
Hence, it follows from elementary combinatorics that 
$$
\sum_{m\ge 0}D_{mn}z^m=\frac{(t+t^{-1})(t^{2^{n+1}-1}-t^{-2^{n+1}+1})}{t^{2^{n+1}}-t^{-2^{n+1}}},
$$
where $z=(t+t^{-1})^{-2}$. 
Thus, the Proposition follows from Proposition \ref{inva}. 
\end{proof} 

Propositions \ref{faith1} and \ref{fullf} show that $F$ is fully faithful. 

Finally, note that the object $X_n$ generates $\C_{2n}$. Indeed, some tensor power of $X_n$ 
contains $F({\rm St}_n)=X_{\lbrace1,\dots,n\rbrace}$ as a subquotient. But $X_{\lbrace1,\dots,n\rbrace}$ is projective, and any indecomposable projective is a direct summand in $X_n^{\otimes m}\otimes X_{\lbrace1,\dots,n\rbrace}$ for some $m$, as desired. 

This completes the proof of part (i) of Theorem \ref{maint}. Thus, Theorem \ref{maint} is proved.  

\subsection{Remarks} 

\begin{remark} Observe that all indecomposable objects in $\C_2$ are direct summands of tensor powers of $X_1$, 
which implies that $\mathcal{T}_{2,2}=\C_2$. Hence $\mathcal{T}_{2,2}$ is an abelian category, as was shown by V. Ostrik. 
\end{remark} 

\begin{remark}\label{allsym} We claim that any braiding on $\C_{2n}$ is necessarily symmetric. Indeed, let $X:=X_n$. Since $\C_{2n}$ is tensor generated by $X$, it suffices to show that if $c_{XX}: X\otimes X\to X\otimes X$ is the braiding map then $c_{XX}^2={\rm Id}$. Since $X\otimes X$ is the regular $A_{n-1}$-module, 
we have $\End(X\otimes X)=\k[t]/(t^2)$. Thus $c_{XX}=\alpha+\beta t$, where $\alpha,\beta\in \k$. Thus $c_{XX}^2=\alpha^2$ is a scalar. This means that $c_{X\otimes X,X}^2$ is also a scalar, namely 
$\alpha^4$. But $X\otimes X$ contains $\one$, hence $\alpha^4=1$. This implies that $\alpha=1$, 
hence $c_{XX}^2={\rm Id}$, as desired. 

This gives another way to see that the braiding on $\C_{2n+2}$ constructed in the proof of Theorem \ref{maint} is automatically symmetric. 
\end{remark} 

\begin{remark} Note that we had to use a nondegenerate $R$-matrix on $A_n$; otherwise (i.e., had we used the trivial one) the bimodule category $\M_n$ would not have been invertible. Moreover, it is easy to see that such a nondegenerate $R$-matrix is unique up to an isomorphism, i.e., there are no choices involved in the construction of $\C_i$. Together with Remark \ref{allsym}, this implies that $\C_{2n}$ 
has a unique braiding up to equivalence, which is its symmetric structure. 
\end{remark} 

\begin{remark} Since the Frobenius twist is a monoidal functor, we have $X_S^{(1)}=0$ if $1\in S$ and $X_S^{(1)}=X_{S-1}$ if 
$1\notin S$, where $S-1$ is the image of $S$ under the map $i\mapsto i-1$. 
\end{remark} 

\begin{remark} The category $\C_{2n}$ is filtered by full abelian subcategories $\C_{2n}^r$, $r\ge 0$, which consist of subquotients of direct sums of $X_n^{\otimes j}$, $j\le r$. These categories are not closed under tensor product, but we have partial tensor products $\C_{2n}^j\times \C_{2n}^{r-j}\to \C_{2n}^r$ with associativity isomorphisms satisfying the pentagon relation. Moreover, we claim that the Frobenius twist functor $X\mapsto X^{(1)}$ is a monoidal functor $F\colon \C_{2n}\to \C_{2n-2}$. The proof is by induction in $n$. The base case $n=1$ claiming that $F(\C_2)\subset \C_0=\Vec$ follows from the fact that $F(X_1)=F(P_\one)=0$. 
To prove the inductive step, assume that $F(\C_{2n})\subset \C_{2n-2}$ and let us show that $F(\C_{2n+2})\subset \C_{2n}$. First note that $F(\C_{2n+1})\subset \C_{2n-1}$, since this is true for simple objects, and the composition series of $F({\rm gr}Y)$ dominates the composition series of $F(Y)$; so the statement follows from Proposition \ref{Serre}. Since every indecomposable object 
of $\C_{2n+2}$ which is not in $\C_{2n+1}$ is of the form $X_{n+1}\otimes Y$ for $Y\in \C_{2n}$, and $F(X_{n+1})=X_n$, the statement follows from the induction assumption.  

Unfortunately, the functor $F$ is not a tensor functor, as it is not 
left or right exact. However, we expect that it is exact in any fixed degree $r$ of the filtration and moreover is an equivalence $\C_{2n}^r\to \C_{2n-2}^r$ preserving partial tensor products (i.e., a partial tensor functor) if $n$ is sufficiently large compared to $r$. If so, we can define the limit $\C(\infty)=\lim_{n\to \infty}\C_{2n}:=\ccup_{r\ge 0}\lim_{n\to \infty}\C_{2n}^r$, where $\C_{2n-2}^r$ is identified with $\C_{2n}^r$ for large $n$ by means of the Frobenius twist functor $F$. We expect that $\C(\infty)\cong \Rep SL(2,\k)$.\footnote{This was recently proved in \cite{EHS} after this paper was published.} This is similar to the construction of the abelian envelope of the Deligne category ${\rm Rep}^{\rm ab}(GL_t)$ out of representation categories of the supergroups $GL(n+t|n)$ using the Duflo-Serganova homology functor ${\rm DS}$ in place of the Frobenius functor, \cite{EHS}. 
\end{remark} 

\begin{remark} By Theorem 2.4 of \cite{O1}, in Proposition \ref{uni} the object $X$ has to be self-dual. 
Let us give a direct proof of this fact. Namely, we have 

\begin{proposition}\label{selfdu}
Let $X$ be an object of a symmetric tensor category $\D$ over a field $\k$ of any characteristic such that $\wedge^2X\cong \one$. Then the isomorphism $\one\to \wedge^2X\subset X\otimes X$ defines an isomorphism $X^*\to X$. 
\end{proposition}

\begin{proof} If $X$ is simple, there is nothing to prove. Otherwise, let $Y\subset X$ be a nonzero subobject such that $Z\colon =X/Y$ is nonzero. Then the composition series of $\wedge^2X$ contains the union of composition series of $\wedge^2 Y$, $Y\otimes Z$ and $\wedge^2Z$. Since $Y\otimes Z\ne 0$, we have 
$Y\otimes Z=\one$. Thus $Y$ is invertible and $Z\cong Y^*$. Now it is easy to see that the map $\gamma: X^*\to X$ corresponding to the isomorphism $\one\to \wedge^2X$ is an isomorphism, as claimed. 
\end{proof} 
\end{remark} 

\section{Further properties of the categories $\C_n$.}

\subsection{The matrix of multiplication by $X_n$ on ${\rm Gr}(\C_{2n})$}

\begin{proposition} Let $B_n\in {\rm Mat}_{2^n}(\ZZ_{\ge 0})$ 
be the matrix of multiplication by $X_n$ on the Grothendieck ring ${\rm Gr}(\C_{2n})$ (in the basis of simple objects). 
Then the matrices $B_n$ are computed recursively as  follows: $B_0=0$ and 
$$
B_{n+1}=\left(\begin{matrix} 0 & 2+B_n\\ 1& 0\end{matrix}\right),
$$
where the blocks are of size $2^{n-1}$. This matrix has distinct eigenvalues 
$\zeta_n^{2r+1}+\zeta_n^{-2r-1}$, $r=1,\dots,2^n$, where 
$\zeta_n=\exp(\pi i/2^{n+1})$. 
\end{proposition} 

\begin{proof} We have $X_nX_S=X_{S\cup n}$ if $n\notin S$ and 
$$
X_nX_S=X_n^2X_{S\setminus n}=(2+X_{n-1})X_{S\setminus n}.
$$ 
This implies the first statement. Also $B_n$ has an eigenvalue $\FPdim(X_n)=2\cos(\pi/2^{n+1})$, 
so the second statement follows from the Galois group action. 
\end{proof} 

\subsection{The Cartan matrix of $\C_n$}

Let $C_n$ be the Cartan matrix of $\C_n$. 
Then by Theorem \ref{maint}(ix), 
$$
C_{2n}=C_{2n-1}\oplus C_{2n-3}\oplus\cdots\oplus C_1\oplus  C_0,
$$
with $C_0=1$, so it suffices to determine $C_n$ for odd $n$. 

\begin{proposition} 
The matrices $C_{2n+1}$ are determined from the recursion
$$
C_{2n+1}=\left(\begin{matrix} 2C_{2n-1} & C_{2n-1}\\ C_{2n-1}& 2C_{2n-2}\end{matrix}\right)
$$
with $C_1=2$. In particular, $C_{2n+1}$ is symmetric and all its nonzero entries are powers of $2$. 
\end{proposition} 

\begin{proof} First of all, $C_n$ is symmetric for all $n$ since the distinguished invertible object 
of $\C_n$ is $\one$ (the only invertible object of $\C_n$), hence all the projective objects 
of $\C_n$ (which are also injective) are self-dual, since so are all the simple objects (see \cite{EGNO}, Subsection 6.4). 

Let $Y_j$ be the simple objects of $\C_{2n-2}$ (and $\C_{2n-1}$). Then $P_{2n-1}(Y_j)=P_{2n}(Y_j)$ and 
$P_{2n+1}(Y_j)=A_n\otimes P_{2n-1}(Y_j)$. Hence 
$$
\Hom_{\C_{2n+1}}(P_{2n+1}(Y_j),P_{2n+1}(Y_r))=
\Hom_{\C_{2n+1}}(A_n\otimes P_{2n-1}(Y_j),A_n\otimes P_{2n-1}(Y_r))=
$$
$$
\Hom_{\C_{2n}}(P_{2n-1}(Y_j),A_n\otimes P_{2n-1}(Y_r))=
$$
$$
\Hom_{\C_{2n}}(P_{2n-1}(Y_j),P_{2n-1}(Y_r))^{\oplus 2}\oplus \Hom_{\C_{2n}}(P_{2n-1}(Y_j),X_n\otimes P_{2n-1}(Y_r))=
$$
$$
\Hom_{\C_{2n}}(P_{2n-1}(Y_j),P_{2n-1}(Y_r))^{\oplus 2}=
$$
$$
\Hom_{\C_{2n-1}}(P_{2n-1}(Y_j),P_{2n-1}(Y_r))^{\oplus 2},
$$
where we use that $A_n=\one\oplus X_n\oplus \one$. This implies that the left upper block of $C_{2n+1}$ is $2C_{2n-1}$. 
Similarly, 
$$
\Hom_{\C_{2n+1}}(P_{2n+1}(X_n\otimes Y_j),P_{2n+1}(Y_r))=
$$
$$
\Hom_{\C_{2n+1}}(A_n\otimes P_{2n-1}(X_n\otimes Y_j),A_n\otimes P_{2n-1}(Y_r))=
$$
$$
\Hom_{\C_{2n}}(P_{2n-1}(X_n\otimes Y_j),A_n\otimes P_{2n-1}(Y_r))=
$$
$$
\Hom_{\C_{2n}}(X_n\otimes P_{2n-2}(Y_j),A_n\otimes P_{2n-1}(Y_r))=
$$
$$
\Hom_{\C_{2n}}(X_n\otimes P_{2n-2}(Y_j),X_n\otimes P_{2n-1}(Y_r))=
$$
$$
\Hom_{\C_{2n-2}}(P_{2n-2}(Y_j),P_{2n-1}(Y_r))=
$$
$$
\Hom_{\C_{2n-1}}(P_{2n-1}(Y_j),P_{2n-1}(Y_r)),
$$
which yields that the upper right and the lower left blocks of $\C_{2n+1}$ are both $C_{2n-1}$. 
Finally,  
$$
\Hom_{\C_{2n+1}}(P_{2n+1}(X_n\otimes Y_j),P_{2n+1}(X_n\otimes Y_r))=
$$
$$
\Hom_{\C_{2n+1}}(A_n\otimes P_{2n-1}(X_n\otimes Y_j),A_n\otimes P_{2n-1}(X_n\otimes Y_r))=
$$
$$
\Hom_{\C_{2n}}(P_{2n-1}(X_n\otimes Y_j),A_n\otimes P_{2n-1}(X_n\otimes Y_r))=
$$
$$
\Hom_{\C_{2n}}(X_n\otimes P_{2n-2}(Y_j),A_n\otimes P_{2n-1}(X_n\otimes Y_r))=
$$
$$
\Hom_{\C_{2n}}(X_n\otimes P_{2n-2}(Y_j),P_{2n-1}(X_n\otimes Y_r))^{\oplus 2}=
$$
$$
\Hom_{\C_{2n}}(X_n\otimes P_{2n-2}(Y_j),X_n\otimes P_{2n-2}(Y_r))^{\oplus 2}=
$$
$$
\Hom_{\C_{2n-2}}(P_{2n-2}(Y_j),P_{2n-2}(Y_r))^{\oplus 2},
$$
which implies that the lower right block of $C_{2n-1}$ is $2C_{2n-2}$, as claimed. 
\end{proof} 

\subsection{The incompressibility of $\C_{2n}$}

\begin{definition} (V. Ostrik) A tensor category $\C$ is called incompressible if any 
tensor functor $F\colon \C\to \D$ from $\C$ to a tensor category $\D$ is injective (i.e., a fully faithful embedding). 
\end{definition} 

\begin{theorem} The category $\C_{2n}$ is incompressible. 
\end{theorem} 

\begin{proof} The proof is by induction in $n$. The base $n=0$ is clear, so we only need to do the inductive step. 
Let $F\colon \C_{2n}\to \D$ be a tensor functor. We need to show that $F$ is fully faithful. 
By replacing $\D$ with ${\rm Im}F$, we may assume that $\D$ is finite and $F$ is surjective. 
Also by the induction assumption we may assume that $\C_{2n-2}\subset \D$ and $F|_{\C_{2n-2}}={\rm Id}$. 

Let $V_S=F(X_S)$. By the induction assumption, $V_S=X_S$ is simple if $n\notin S$.
Also it is clear that $V_n$ is simple (as its FP dimension is $<2$). 

We claim that $\D$ does not contain nontrivial invertible objects. Indeed, if $\chi$ is an invertible object then 
$\chi$ must occur as a composition factor in $V_n^{\otimes r}$ for some $r$. 
This means that $V_n\otimes \chi$ occurs in $V_n^{\otimes r+1}$, 
which is impossible if $r$ is odd for FP dimension reasons, since $V_n^{\otimes r+1}$ is a 
power of $V_n\otimes V_n$, hence its composition factors are $X_S$, $n\notin S$. Thus 
$r$ is even, and $\chi\in \C_{2n-2}$, hence $\chi=\one$. 
 
Consider the tensor product $V_n\otimes V_n=F(A_{n-1})$. 
As $V_n$ is simple, we have 
$$
\Hom(V_n\otimes V_n,\one)=
\Hom(\one,V_n\otimes V_n)=\k. 
$$
Also, we claim that 
$$
\Hom(V_n\otimes V_n,X_{n-1})=
\Hom(X_{n-1},V_n\otimes V_n)=0. 
$$
Indeed, otherwise $V_n$ is a composition factor of $X_{n-1}\otimes V_n$, 
so we have an object $Y\in \D$ with dimension 
$$
\FPdim(Y)=\FPdim(X_n)\FPdim(X_{n-1})-\FPdim(X_n)=
$$
$$
=(\zeta_n+\zeta_n^{-1})(\zeta_n^2+\zeta_n^{-2})-(\zeta_n+\zeta_n^{-1})=\zeta_n^3+\zeta_n^{-3}.
$$
But this is impossible, since this number has a larger Galois conjugate (namely, $\zeta_n+\zeta_n^{-1}$), while the Frobenius-Perron dimension of an object must be the largest element in its Galois orbit.  
Thus, we see that $V_n\otimes V_n$ is indecomposable with composition series 
$\one,X_{n-1},\one$. 

Now, we claim that $\Hom(V_{\lbrace{1,\dots,n\rbrace}},V_{\lbrace{1,\dots,n\rbrace}})=\k$. 
Indeed, we can write this space as 
$$
\Hom(V_n\otimes V_n,X_{\lbrace{1,\dots,n-1\rbrace}}\otimes X_{\lbrace{1,\dots,n-1\rbrace}})=
$$
$$
\Hom(V_n\otimes V_n,X_{\lbrace{1,\dots,n-2\rbrace}}\otimes A_{n-2}\otimes X_{\lbrace{1,\dots,n-2\rbrace}}).
$$
Since $X_{n-1}$ does not occur in $X_{\lbrace{1,\dots,n-2\rbrace}}\otimes A_{n-2}\otimes X_{\lbrace{1,\dots,n-2\rbrace}}$, this space equals 
$$
\Hom_{\C_{2n-2}}(\one, X_{\lbrace{1,\dots,n-2\rbrace}}\otimes A_{n-2}\otimes X_{\lbrace{1,\dots,n-2\rbrace}})=
$$
$$
\Hom(X_{\lbrace{1,\dots,n-1\rbrace}},X_{\lbrace{1,\dots,n-1\rbrace}})=\k, 
$$
as desired. 

This means that $V_{\lbrace{1,\dots,n\rbrace}}$ is indecomposable. Also, it is projective since 
$X_{\lbrace{1,\dots,n\rbrace}}$ is projective, and $F$ is surjective (\cite{EGNO}, Theorem 6.1.16). Since $\D$ has no nontrivial invertible objects, 
the head and socle of $V_{\lbrace{1,\dots,n\rbrace}}$ are isomorphic. 
Hence $V_{\lbrace{1,\dots,n\rbrace}}$ is simple (as its endomorphism algebra is 1-dimensional). 
This implies that $V_S$ is simple for all $S$ (as it is a tensor factor of $V_{\lbrace{1,\dots,n\rbrace}}$). 
Thus, the simple objects of $\D$ are the objects $V_S$. 

Now let $P_S$ be the projective cover of $X_S$ in $\C_{2n}$, $Q_S$ the projective cover of $V_S$ in $\D$.
Let $p_S=\FPdim(P_S)$ and $q_S=\FPdim(Q_S)$. 
Let $\mathbf p=(p_S), \mathbf q=(q_S)$. Then $\mathbf p,\mathbf q$ are left Frobenius--Perron eigenvectors 
of the matrix of multiplication by $V_n$. Hence $\mathbf p=\lambda \mathbf q$  for some $\lambda>0$. 
But we have shown that $p_{\lbrace{1,\dots,n\rbrace}}=q_{\lbrace{1,\dots,n\rbrace}}$, hence $\lambda=1$. 
Thus $F(P_S)=Q_S$ for all $S$, and $F$ is an equivalence. 
\end{proof} 

\subsection{Connection with modular representation theory of finite groups}

The category ${\mathcal T}_{n,p}$ arises in the modular representation theory of the group $SL(2,\Bbb F_q)$, where $q=p^n$.

Namely, let $\k$ be an algebraically closed field of characteristic $p$, and pick an embedding $\Bbb F_q\hookrightarrow \k$. Let $V=\k^2$ be the 2-dimensional tautological representation of 
$SL(2,\Bbb F_q)$. Note that $V$ is self-dual. Let $\widetilde{\D}_{n,p}$ be the full subcategory of $\Rep_\k SL(2,\Bbb F_q)$ whose objects are finite direct sums of direct summands of tensor powers of $V$. It is clear that $\widetilde{\D}_{n,p}$ is a rigid monoidal Karoubian category. Moreover, since all indecomposable objects $T_i$ of ${\rm Tilt}_p$ occur as direct summands in tensor powers of the $2$-dimensional module $T_1$, the natural restriction functor ${\rm Res}\colon {\rm Tilt}_p\to \Rep_\k SL(2,\Bbb F_q)$ in fact lands in $\widetilde{\D}_{n,p}$, and the indecomposable objects of $\widetilde{\D}_{n,p}$ are the direct summands in ${\rm Res}(T_i)$, $i\ge 0$. 
 
Recall that ${\rm St}_n$ denotes the $n$-th Steinberg module, ${\rm St}_n=T_{q-1}$. The restriction of ${\rm St}_n$ 
to the group algebra of the $p$-Sylow subgroup of $SL(2,\Bbb F_q)$ is free (of rank 1), hence ${\rm St}_n$ generates 
the tensor ideal $\mathcal {P}$ of projective modules in $\Rep_\k SL(2,\Bbb F_q)$. Moreover, since 
$$
T_1\otimes T_i=T_{i+1}\oplus \bigoplus_{s<i}c_{is}T_s,
$$ 
this ideal contains ${\rm Res}(T_i)$ for $i\ge q-1$, i.e., ${\rm Res}(\mathcal{I}_n)\subset \mathcal{P}$. 

Let $\D_{n,p}:=\widetilde{\D}_{n,p}/\mathcal{P}$ be the stable category of $\widetilde{\D}_{n,p}$. 
It follows that ${\rm Res}$ descends to a symmetric monoidal functor $F\colon {\mathcal T}_{n,p}\to {\mathcal D}_{n,p}$. 
It also follows that the indecomposable objects of $\D_{n,p}$ are the indecomposable (non-projective) direct summands of $V^{\otimes r}$, $r<q-1$ (we will see later that in fact there are no projective direct summands). 

Moreover, since ${\rm Res}({\rm St}_{n-1})$ is not projective (as its dimension $p^{n-1}$ is less than the order $q=p^n$ of a $p$-Sylow subgroup of $SL(2,\Bbb F_q)$), we have $F({\rm St}_{n-1})\ne 0$, hence by Lemma \ref{faith} 
$F$ is faithful. 

\begin{proposition}\label{equiva} The functor $F$ is an equivalence of categories.  
\end{proposition} 

\begin{proof} We will use the following lemma. 

\begin{lemma}\label{invar} If $r,s<q$ then any homomorphism of $SL(2,\Bbb F_q)$-modules $A\colon V^{\otimes r}\to V^{\otimes s}$ is in fact a homomorphism of $SL(2,\k)$-modules. 
\end{lemma}

\begin{proof} Let $U$ be a maximal unipotent subgroup of $SL(2,\k)$. Then the condition that $A$ commutes with 
$u\in \k\cong U(\k)$ is a system of polynomial equations with respect to $u$ of degree ${\rm max}(r,s)<q$. This system is satisfied for any $u\in \Bbb F_q$. But any polynomial of degree $<q$ which vanishes on $\Bbb F_q$ has to vanish on $\k$. Thus these equations are satisfied for any $u\in \k$, hence $A$ commutes with $U(\k)$. But $SL(2,\k)$ is generated by its subgroups of the form $U(\k)$, which implies the statement. 
\end{proof} 

\begin{corollary}\label{inde} If $r<q$ then the module ${\rm Res}(T_r)$ is indecomposable. 
\end{corollary} 

\begin{proof} Since $T_r$ is a direct summand in $T_1^{\otimes r}$, Lemma \ref{invar} implies that the functor ${\rm Res}$ 
induces an isomorphism ${\rm End}(T_r)\cong {\rm End}({\rm Res}(T_r))$. Hence ${\rm End}({\rm Res}(T_r))$ is a local algebra, which implies the statement. 
\end{proof} 

Corollary \ref{inde} implies that the functor $F$ is essentially surjective, i.e., the indecomposable 
objects of $\D_{n,p}$ are $V_r:={\rm Res}(T_r)$, $r<q-1$. Note that these modules are not projective, since $F$ is faithful. 

It remains to show that $F$ is full. To this end, let $f\colon V_r\to V_s$ be a morphism in $\D_{n,p}$ ($r,s<q-1$). 
Let $\widetilde{f}\colon V_r\to V_s$ be a preimage of $f$ in $\widetilde{\D}_{n,p}$. 
By Lemma \ref{invar}, $\widetilde{f}$ may be viewed as a morphism $T_r\to T_s$. 
Let $\widehat{f}$ be its image in ${\mathcal T}_{n,p}$. It is clear that $F(\widehat{f})=f$, which implies the statement. 
Proposition \ref{equiva} is proved. 
\end{proof} 

Proposition \ref{equiva} and Theorem \ref{maint}(i) imply

\begin{corollary}\label{embe} The stable category $\D_{n+1,2}$ embeds as a full monoidal subcategory into the abelian symmetric tensor category $\C_{2n}$. 
\end{corollary}

Note that this is a rather striking property, as full monoidal subcategories of stable categories do not normally admit an abelian envelope, i.e., an additive monoidal embedding into an abelian monoidal category. For example, we have the following proposition. 

Recall that an indecomposable finite dimensional representation $E$ of a finite group $G$ over $\k$ with ${\rm char}(\k)$ dividing the order of $G$ is {\it endotrivial} if $E\otimes E^*=\one \oplus P$, where $P$ is projective; in other words, $E$ is invertible in the stable category of $G$. 

\begin{proposition}\label{nonemb} Let $G\ne 1$ be a finite $p$-group and $\mathcal S_\k(G)$ be the stable category of $\Rep_\k G$, where ${\rm char}(\k)=p$. Let  $E\ne \k$ be any endotrivial indecomposable representation of $G$. Then the full tensor subcategory $\mathcal E$ of $\mathcal S_\k(G)$ generated by $E$ does not admit an abelian envelope. 
\end{proposition} 

\begin{proof} Let $P$ be a free $\k G$-module of finite rank, $V$ a non-projective indecomposable finite dimensional 
$G$-module, and $g\colon V\to P$, $h\colon P\to \k$ be morphisms. We claim that $h\circ g=0$. Indeed, assume the contrary. 
We may assume that $P$ has rank $1$. Then ${\rm Im}(g)=P$, hence $V=P\oplus V'$, a contradiction. 

Let $f\colon E\to \k$ be a nonzero homomorphism (it exists since $G$ is a $p$-group). It follows that
$f$ remains nonzero in the stable category. Since $E$ is invertible (hence simple) and $E\ne \k$, this means that 
$\mathcal{E}$ does not have an abelian envelope. 
\end{proof} 

\begin{example} For example, Proposition \ref{nonemb} applies when $G=\Bbb Z/p$ with $p>2$ or $G=(\Bbb Z/2)^2$ and $E=\Omega(\k)$. 
\end{example} 

Let $B$ be a Borel subgroup of $SL(2,\k)$ (the subgroup of upper triangular matrices). Let $\widetilde{\mathcal B}_{n,p}$ 
be the full subcategory of $\Rep_\k B(\Bbb F_q)$ whose objects are finite direct sums of direct summands of $V^{\otimes r}$. Let $\mathcal B_{n,p}$ be the stable category of $\widetilde{\mathcal B}_{n,p}$, i.e., its quotient by the tensor ideal of projective objects. Since $B(\Bbb F_q)$ is the normalizer of a Sylow subgroup of $SL(2,\Bbb F_q)$, by Green's correspondence (see \cite{Fe}, Section III.5), the restriction functor defines an equivalence $\D_{n,p}\cong \mathcal B_{n,p}$. Moreover, we have the following proposition. 

Let $U=[B,B]$ and $\widetilde{\mathcal U}_{n,p}$ be the full subcategory of $\Rep_\k U(\Bbb F_q)$ whose objects are finite direct sums of direct summands in tensor powers of $V$. Let $\mathcal U_{n,p}$ be the corresponding stable category. 
We have the symmetric monoidal restriction functor $H\colon \mathcal B_{n,p}\to \mathcal U_{n,p}$. 

\begin{proposition}\label{bij} The functor $H$ maps indecomposable objects to indecomposable ones, and defines a bijection between isomorphism classes of objects of $\mathcal T_{n,p}=\mathcal D_{n,p}=\mathcal B_{n,p}$ and 
$\mathcal U_{n,p}$. In particular, tensor powers of the $U(\Bbb F_q)=(\Bbb Z/p)^n$-module $V$ contain exactly $q-1$ distinct non-projective indecomposable direct summands (the restrictions of $T_r$, $0\le r<q-1$). 
\end{proposition} 

\begin{proof} To show that $H$ maps indecomposable objects to indecomposable ones, 
it suffices to show that the restriction of $T_r$ to $U(\Bbb F_q)$ is indecomposable for $r<q-1$. 
To this end, note first that the restriction of $T_r$ to $B$ is indecomposable, since it is already so for $B(\Bbb F_{p^N})$ for large $N$. Namely, we have shown that it is so modulo projectives, but for large $N$ the restriction of $T_r$ is too small to contain projective direct summands. 

Now we claim that the restriction of any indecomposable finite dimensional rational $B$-module $Y$ to $U$ is indecomposable. Indeed, $\End_B(Y)=\End_U(Y)^{B/U}$, where $B/U=\Bbb G_m$. Let $\bold e$ 
be a primitive central idempotent of the semisimple algebra $\End_U(Y)/{\rm Rad}(\End_U(Y))$. 
Then $\bold e$ is $B/U$-invariant, hence it belongs to 
$$
(\End_U(Y)/{\rm Rad}(\End_U(Y))^{B/U}=\End_U(Y)^{B/U}/({\rm Rad}(\End_U(Y)))^{B/U}.
$$
Hence $\bold e$ can be lifted to an idempotent $\widetilde{\bold e}$ in $\End_U(Y)^{B/U}=\End_B(Y)$.
Since $Y$ is indecomposable, we have $\widetilde{\bold e}=1$, hence $\bold e=1$. Thus, $Y|_U$ is indecomposable, as desired. 

Thus, $T_r|_U$ is indecomposable. Now using the argument in the proof of Lemma \ref{invar}, we conclude that 
the restriction map $\End_{U(\k)}(T_r)\to \End_{U(\Bbb F_q)}(T_r)$ is an isomorphism. 
Hence the algebra $\End_{U(\Bbb F_q)}(T_r)$ is local, since so is $\End_{U(\k)}(T_r)$, as desired. 

It remains to show that if $T_r|_{U(\Bbb F_q)}\cong T_s|_{U(\Bbb F_q)}$ 
for $0\le r,s<q-1$ then $r=s$. 
By Lemma \ref{invar}, if $T_r|_{U(\Bbb F_q)}\cong T_s|_{U(\Bbb F_q)}$ then we have an isomorphism 
$\phi: T_r|_{U(\k)}\cong T_s|_{U(\k)}$. This isomorphism 
is, in general, not a $B$-isomorphism, so it defines a regular 1-cocycle $f(z)\colon =z(\phi)^{-1}\circ \phi$
of $\Bbb G_m=B/U$ with values in ${\rm Aut}_U(T_r)$. This cocycle has the form 
$f(z)=z^mg(z)$, where $g: \Bbb G_m\to {\rm Aut}_U^1(T_r)$ is a regular 1-cocycle of $\Bbb G_m$ with values in the unipotent part of ${\rm Aut}_U^1(T_r)$ of ${\rm Aut}_U(T_r)$. 

We claim that $g$ must be a coboundary. To show this, it suffices to show that any regular 1-cocycle 
$\Bbb G_m\to \Bbb G_a$ for any action of $\Bbb G_m$ on $\Bbb G_a$  is a coboundary 
(as ${\rm Aut}_U^1(T_r)$ has a filtration whose successive quotients 
are $\Bbb G_a$). But such a cocycle is just a Laurent polynomial 
$h(z)$ such that 
$$
h(ab)=h(a)+a^nh(b)
$$
for some $n$. It is easy to show explicitly that such a polynomial must be of the form 
$h(z)=c(z^n-1)$, i.e., a coboundary, as claimed. 

Thus, modifying $\phi$ by an element of ${\rm Aut}_U^1(T_r)$, we may 
assume that $z(\phi)=z^m\phi$ for some $m\in \Bbb Z$. But the characters 
of $T_r,T_s$ are symmetric with respect to the map $x\to x^{-1}$, which implies that we must have $m=0$. 
Thus, $\phi$ is an isomorphism of $B$-modules, i.e., $T_r$, $T_s$ have the same character, so $r=s$, as desired. 
\end{proof}

\subsection{${\rm Ext}^1$ between simple objects}

Let $D_m:=\dim {\rm Ext}^1_{\C_m}(X_S,X_T)$. 

\begin{proposition}\label{extST} 
The numbers $D_m(S,T)$ are zeros and ones and are determined by the formulas
$$
D_{2n+1}(S,T)=D_{2n}(S,T)=D_{2n-1}(S,T),\ n\notin S,T;
$$
$$
D_{2n+1}(S,T)=D_{2n}(S,T)=D_{2n-2}(S\setminus n,T\setminus n),\ n\in S,T;  
$$
$$
D_{2n+1}(S,T)=\delta_{S\setminus n,T},\ n\in S, n\notin T;\quad D_{2n+1}(S,T)=\delta_{S,T\setminus n},\ n\in T, n\notin S; 
$$
and 
$$
D_{2n}(S,T)=0  
$$
if $n$ is contained in exactly one of the sets $S,T$, with initial conditions $D_1(\varnothing,\varnothing)=1$ and $D_0(\varnothing,\varnothing)=0$.  
\end{proposition} 

\begin{proof} Since $\C_{2n}=\C_{2n-1}\oplus \C_{2n-2}\otimes X_n$, we obtain the statement for $D_{2n}(S,T)$. 
Also, we have 
$$
D_{2n+1}(S,T)=\dim {\rm Ext}^1_{A_n}(X_S,X_T)=D_{2n}(S,T)+\dim \Hom(X_n\otimes X_S,X_T), 
$$
which implies the statement about $D_{2n+1}(S,T)$. 
\end{proof} 

\begin{corollary} Let $S,T\subset \Bbb Z_{>0}$ be finite subsets with ${\rm max}(S\cup 0)=s$ and ${\rm max}(T\cup 0)=t$. Let $n=\max(s,t)$. 
Then $D_m(S,T)$ are the same for all $m\ge 2n+1$. Moreover, if $s=n>t$ then we have 
$$
D_{2n+1}(S,T)=\delta_{S\setminus n,T}, D_{2n}(S,T)=0,
$$
and if $s<t=n$ then 
$$
D_{2n+1}(S,T)=\delta_{S,T\setminus n}, D_{2n}(S,T)=0. 
$$
Further, if $s=t=n>0$ then $D_{2n+1}(S,T)=D_{2n}(S,T)=0$ unless $S=T$, 
and $D_{2n+1}(S,S)=D_{2n}(S,S)$ equals $0$ if $1\in S$ and $1$ if $1\notin S$. 
In particular, $D_m(\varnothing,\varnothing)=1$ if $m>0$ and $0$ if $m=0$. 
\end{corollary} 

\begin{proof}
The corollary follows by straightforward application of Proposition \ref{extST}. 
\end{proof} 

\subsection{The universal property of $\C_{2n}$}

\begin{theorem}\label{unipro}
Let $\D$ be a symmetric tensor category over $\k$. 
Then the following statements hold for $n\ge 1$. 

(i) Any faithful $\k$-linear symmetric monoidal functor $H: \mathcal{T}_{n+1,2}\to \D$ factors through $\C_{2n}$ in a unique way: $H=G\circ F$ where $F: \mathcal{T}_{n+1,2}\to \C_{2n}$ is the inclusion of Theorem \ref{maint}, and $G\colon  \mathcal{C}_{2n}\to \D$ is a tensor functor. In other words, $\C_{2n}$ is the canonical abelian envelope of $\mathcal{T}_{n+1,2}$ in the sense of Deligne, \cite{De3} (see also \cite{EHS}). 

(ii) Symmetric tensor functors $G\colon \mathcal{C}_{2n}\to \D$ correspond to objects 
$X\in \D$ such that $\wedge^2X\cong \one$, $\wedge^3X=0$, and $Q_{n+1}^+(X)\cong Q_{n+1}^-(X)$, 
but $Q_n^+(X)\ncong Q_n^-(X)$, via $F\mapsto F(X_n)$.
\end{theorem} 

\begin{proof} (i) Recall that $F$ is fully faithful by Theorem \ref{maint}(i). Let 
us identify $\mathcal{T}_{n+1,2}$ with its image in $\C_{2n}$ under $F$.
By Theorem 9.2.2 of \cite{EHS}, it suffices to show that\footnote{Theorem 9.2.2 of \cite{EHS} is stated for $\k=\Bbb C$, but this assumption is not used in the proof, so the theorem applies to any algebraically closed field.}

(1) any object $Y$ of $\C_{2n}$ is the image of a morphism in $\mathcal{T}_{n+1,2}\subset \C_{2n}$, and 

(2) for any epimorphism $f: Y\to Z$ in $\C_{2n}$, there exists a nonzero $T\in \mathcal{T}_{n+1,2}\subset \C_{2n}$ such that the epimorphism $f\otimes 1: Y\otimes T\to Z\otimes T$ splits. 

To establish (1), let $P_Y$ and $I_Y$ be the projective cover and injective hull of $Y$, and let $f: P_Y\to I_Y$ 
be the composition of the natural morphisms $P_Y\to Y\to I_Y$. Then $Y={\rm Im}(f)$, and $f$ is a morphism 
in $\mathcal{T}_{n+1,2}$ (as $P_Y,I_Y$ are both projective and thus contained in $\mathcal{T}_{n+1,2}$). 

To establish (2), it suffices to take $T$ to be any nonzero projective object in $\C_{2n}$. This proves (i). 

(ii) By Corollary \ref{uni1}, symmetric monoidal functors $H: \mathcal{T}_{n+1,2}\to \D$ correspond to objects $X\in \D$ such that $\wedge^2X\cong \one$, $\wedge^3X=0$, and $Q_{n+1}^+(X)\cong Q_{n+1}^-(X)$. By Proposition \ref{faith1}, such $H$ is faithful if and only if $F(T_{2^n-1}) \ne 0$, i.e., $Q_n^+(X)\ncong Q_n^-(X)$. Thus (i) implies (ii). 
\end{proof} 

\section{The structure of the categories $\C_n$ for small $n$.} 

In this section we describe the structure of the categories $\C_n$ for $n\le 6$. 
Many of the results below were obtained by a computer calculation using MAGMA (\cite{BCP}).  

\subsection{The structure of $T_m$ for small $m$} 
We begin by describing the structure of the tilting modules $T_m$ for
$SL(2,\k)$, where $\k$ has characteristic two, for
small values of $m$. Let $L_m$ be the simple $SL(2,\k)$-module with highest
weight $m$; by Steinberg's tensor product theorem this is a tensor
product of Frobenius twists of the natural module $V$, indexed by the
ones in the binary expansion of $m$. So for example we have
$L_{13}=V^{(3)}\otimes V^{(2)}\otimes V$. The structure of these
tilting modules is not hard to compute by hand. It may be
found in the paper of Doty and Henke~\cite{Doty/Henke:2005a}, and is
reproduced here for convenience. We have
$T_0=L_0=\k$, the field with trivial action, and
$T_1 = L_1 = V$, the natural module. Next, we have
$T_2 = [0, 2, 0]$. Here, we write $[a,b,c,\dots]$ for a uniserial module with composition
factors $a$, $b$, $c$, \dots, starting with the top composition
factor. We have also abbreviated $L_m$ to $m$, and we note that
$L_2=V^{(1)}$ is the Frobenius twist of $V$, $L_3=V\otimes V^{(1)}$,
and so on. Continuing this way, here
are the structures of the first few tilting modules.\par\vspace{-7mm}
{\tiny
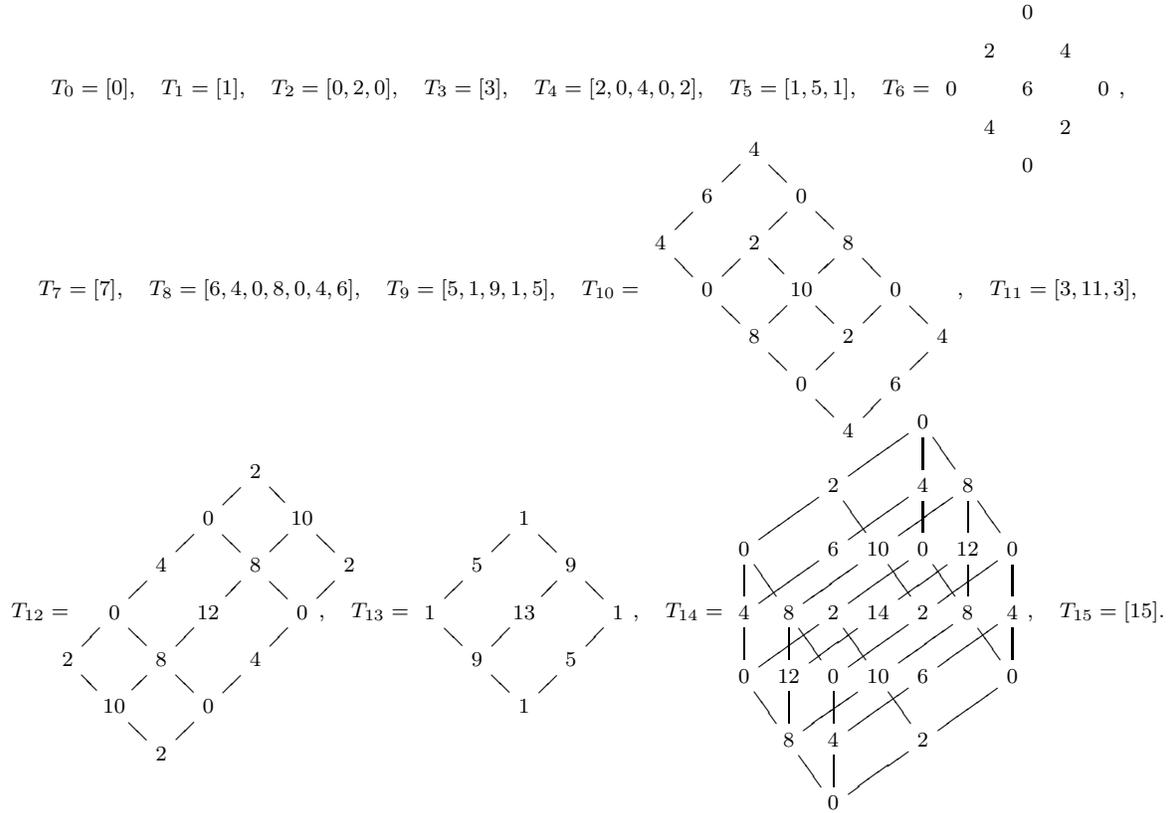
\begin{gather*}
T_0=[0],\quad T_1=[1],\quad T_2=[0,2,0],\quad T_3=[3],\quad
T_4=[2,0,4,0,2], \quad T_5=[1,5,1], \quad
T_6=\vcenter{\xymatrix@=0.8mm@!{&&0\ar@{-}[dl]\ar@{-}[dr] \\ 
&2\ar@{-}[dl]&&4\ar@{-}[dl]\ar@{-}[dr] \\ 
0\ar@{-}[dr] &&6\ar@{-}[dl] && 0\ar@{-}[dl] \\ 
&4\ar@{-}[dr]&&2\ar@{-}[dl] \\ &&0}},\\[-8mm]
T_7=[7],\quad
T_8=[6,4,0,8,0,4,6],\quad 
T_9=[5,1,9,1,5],\quad
T_{10}=\vcenter{\xymatrix@=0.8mm@!{&&4\ar@{-}[dl]\ar@{-}[dr]\\
&6\ar@{-}[dl]&&0\ar@{-}[dl]\ar@{-}[dr]\\
4\ar@{-}[dr]&&2\ar@{-}[dl]\ar@{-}[dr]&&8\ar@{-}[dl]\ar@{-}[dr]\\
&0\ar@{-}[dr]&&10\ar@{-}[dl]\ar@{-}[dr]&&0\ar@{-}[dl]\ar@{-}[dr]\\
&&8\ar@{-}[dr]&&2\ar@{-}[dl]&&4\ar@{-}[dl]\\
&&&0\ar@{-}[dr]&&6\ar@{-}[dl]\\&&&&4}},\quad
T_{11}=[3,11,3],\\[-7mm]
T_{12}=\hspace{-3mm}\vcenter{\xymatrix@=0.8mm@!{&&&&2\ar@{-}[dl]\ar@{-}[dr]\\
&&&0\ar@{-}[dl]\ar@{-}[dr]&&10\ar@{-}[dl]\ar@{-}[dr]\\
&&4\ar@{-}[dl]&&8\ar@{-}[dl]\ar@{-}[dr]&&2\ar@{-}[dl]\\
&0\ar@{-}[dl]\ar@{-}[dr]&&12\ar@{-}[dl]&&0\ar@{-}[dl]\\
2\ar@{-}[dr]&&8\ar@{-}[dl]\ar@{-}[dr]&&4\ar@{-}[dl]\\
&10\ar@{-}[dr]&&0\ar@{-}[dl]\\
&&2}}\hspace{-6mm},\quad
T_{13}=\vcenter{\xymatrix@=0.8mm@!{&&1\ar@{-}[dl]\ar@{-}[dr] \\
&5\ar@{-}[dl]&&9\ar@{-}[dl]\ar@{-}[dr] \\
1\ar@{-}[dr]&&13\ar@{-}[dl]&&1\ar@{-}[dl]\\
&9\ar@{-}[dr]&&5\ar@{-}[dl]\\&&1}},\quad
T_{14}=\vcenter{\xymatrix@R=3mm@C=0.5mm@!{
&&&& 0\ar@{-}[dll]\ar@{-}[d]\ar@{-}[dr]\\
&& 2\ar@{-}[dll]\ar@{-}[dr]&&4\ar@{-}[dll]\ar@{-}[d]& 8\ar@{-}[dll]\ar@{-}[d]\ar@{-}[dr] \\
 0\ar@{-}[d]\ar@{-}[dr]&&6\ar@{-}[dll]&{10}\ar@{-}[dll]\ar@{-}[dr]
&0\ar@{-}[dll]\ar@{-}[dr]&12\ar@{-}[dll]\ar@{-}[d]& 0\ar@{-}[dll]\ar@{-}[d]\\
4\ar@{-}[d]& 8\ar@{-}[d]\ar@{-}[dr]&2\ar@{-}[dll]\ar@{-}[dr]&
14\ar@{-}[dll]& 2\ar@{-}[dll]&8\ar@{-}[dll]\ar@{-}[dr]&4\ar@{-}[dll]\ar@{-}[d]\\
0\ar@{-}[dr]&12\ar@{-}[d]& 0\ar@{-}[d]&10\ar@{-}[dll]\ar@{-}[dr]&6\ar@{-}[dll]&&0\ar@{-}[dll]\\
&8\ar@{-}[dr]&4\ar@{-}[d]&&2\ar@{-}[dll]\\
&&0}},\quad
T_{15}=[15].
\end{gather*}
}
Here the diagrams for the non-uniserial modules are in the sense of
Alperin~\cite{Alperin:1980b} or Benson and Carlson~\cite{Benson/Carlson:1987a}.

The tensor product $T_n\otimes V$ is isomorphic to $T_{n+1}$ plus a
direct sum of $T_m$ with $m<n$, with multiplicities given in the
following table. This table extends in the obvious way for larger
values of $n$.\par
{\tiny
\renewcommand{\arraystretch}{0.85}
\setlength{\arraycolsep}{1mm}
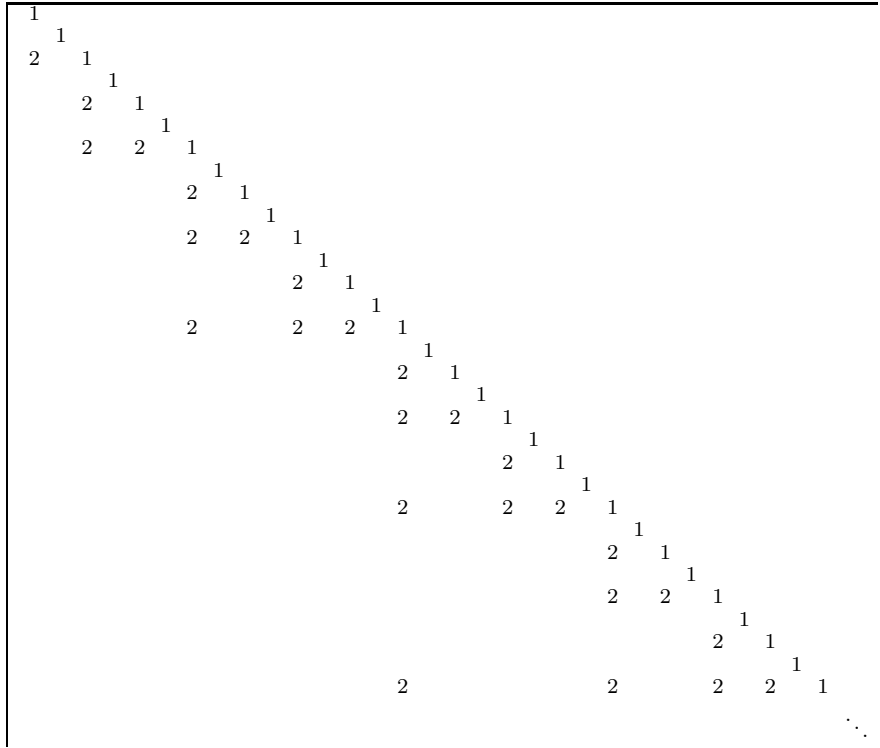
\[ \begin{array}{|cccccccccccccccccccccccccccccccccccccccc}
\hline
&1\\
&&1\\
&2&&1\\
&&&&1\\
&&&2&&1\\
&&&&&&1\\
&&&2&&2&&1\\
&&&&&&&&1\\
&&&&&&&2&&1\\
&&&&&&&&&&1\\
&&&&&&&2&&2&&1\\
&&&&&&&&&&&&1\\
&&&&&&&&&&&2&&1\\
&&&&&&&&&&&&&&1\\
&&&&&&&2&&&&2&&2&&1\\
&&&&&&&&&&&&&&&&1\\
&&&&&&&&&&&&&&&2&&1\\
&&&&&&&&&&&&&&&&&&1\\
&&&&&&&&&&&&&&&2&&2&&1\\
&&&&&&&&&&&&&&&&&&&&1\\
&&&&&&&&&&&&&&&&&&&2&&1\\
&&&&&&&&&&&&&&&&&&&&&&1\\
&&&&&&&&&&&&&&&2&&&&2&&2&&1\\
&&&&&&&&&&&&&&&&&&&&&&&&1\\
&&&&&&&&&&&&&&&&&&&&&&&2&&1\\
&&&&&&&&&&&&&&&&&&&&&&&&&&1\\
&&&&&&&&&&&&&&&&&&&&&&&2&&2&&1\\
&&&&&&&&&&&&&&&&&&&&&&&&&&&&1\\
&&&&&&&&&&&&&&&&&&&&&&&&&&&2&&1\\
&&&&&&&&&&&&&&&&&&&&&&&&&&&&&&1\\
&&&&&&&&&&&&&&&2&&&&&&&&2&&&&2&&2&&1\\
&&&&&&&&&&&&&&&&&&&&&&&&&&&&&&&&\ddots
\end{array} \]
}

\subsection{The structure of $\C_n$ for small $n$}
We now describe the structure of $\C_n$ with $n$ small.

\subsubsection{The categories $\C_0$ and $\C_1$}
We have $\C_0=\mathrm{Vec}_k$, with just one simple object $\one$, and no self
extensions.\bigskip

Next, $\C_1$ again has just one simple object $\one$, but
its projective cover is an extension
$P_\one=\vcenter{\xymatrix@=1mm{\one\lk[d]\\\one}}$.
The cohomology ring is
\[ \Ext^*_{\C_1}(\one,\one) = k[x] \]
where $|x|=1$. 

This is the only case where there is a forgetful tensor
functor to vector spaces; but it is not a symmetric
tensor functor. An object in $\C_1$ can be thought of as a vector
space $V$ together with a linear map $d\colon V \to V$ satisfying $d^2=0$. 
The tensor product is as usual, with $d \colon V\otimes W \to V
\otimes W$ given by $d(v\otimes w)=dv\otimes w + v\otimes dw$. The
commutativity isomorphism is given by $s(v \otimes w) = w \otimes v + dw
\otimes dv$. Thus for example a commutative algebra in $\C_1$ is a
differential algebra satisfying $ab-ba=da\cdot db$.\bigskip

\subsubsection{The category $\C_2$}
In the next case, $\C_2$ has two simple objects, $\one=X_\varnothing$ and $V=X_1$. The
projective cover of $\one$ is as in $\C_1$, and $V$ is a projective
simple object. We have $V\otimes V=P_\one$.  The cohomology ring is
the same as for $\C_1$. The Frobenius map acts on simple objects via 
$\one^{(1)}=\one$ and $V^{(1)}=0$. The Frobenius--Perron dimensions are given by
$\FPdim(\one)=1$ and $\FPdim(V)=\sqrt{2}$.\bigskip

\subsubsection{The category $\C_3$}
The category $\C_3$ also has two simple objects, $\one$ and $V$,
with the same tensor products (except that $P_\one$ is interpreted in
$\C_2$), the same Frobenius--Perron dimensions,
and the same Frobenius map as in $\C_2$, but
the projective covers are more complicated. They are given by
\[ P_\one=\vcenter{\xymatrix@=3mm{&\one\lkdl\lkdr \\ V\lk[d] & & \one\lk[d] \\ 
\one\lkdr & & V\lkdl \\& \one}}\qquad P_V=
\vcenter{\xymatrix@=3mm{V \lk[d] \\ \one \lk[d] \\ \one \lk[d] \\
    V}}\quad . \]
The cohomology ring is
\[ \Ext^*_{\C_3}(\one,\one)=k[x,y,z]/(y^2+xz) \]
where $|x|=1$, $|y|=2$, $|z|=3$. The Poincar\'e series of this ring is
given by
\[ \sum_i  t^i\dim\Ext^i_{\C_3}(\one,\one)=\frac{1+t^2}{(1-t)(1-t^3)}. \]

One can also compute the Ext ring between $V$ and itself. 
To this end, note that the minimal
resolution of $V$ is
$$
\cdots \to P_V \to P_\one \to P_V \to P_V \to P_\one \to P_V \to V \to 0
$$
and is periodic with period $3$. Using this, one shows that 
\[ \Ext^*_{\C_3}(V,V)=k[u,v]/(u^2) \]
where $|u|=2$, $|v|=3$. The Poincar\'e series of this ring is
\[ \sum_i t^i \dim \Ext^i_{\C_3}(V,V) = \frac{1+t^2}{1-t^3}. \]

Finally, we can compute ${\rm Ext}_{\C_3}^*(V,\one)$ and ${\rm Ext}_{\C_3}^*(\one,V)$ 
as bimodules over ${\rm Ext}_{\C_3}^*(\one,\one)$ and ${\rm Ext}_{\C_3}^*(V,V)$. Namely, using the above resolution, we see that 
${\rm Ext}_{\C_3}^*(V,\one)$ has a generator in degree one,
and is annihilated by $x$ and $y$ in ${\rm Ext}_{\C_3}^*(\one,\one)$ and by $u$ in ${\rm Ext}_{\C_3}^*(V,V)$. The elements
$z$ and $v$ both act as the periodicity generator in degree $3$. For ${\rm Ext}_{\C_3}^*(\one,V)$ we 
have exactly the same structure. In both cases, the Poincar\'e series is $t/(1-t^3)$.

\subsubsection{The category $\C_4$} 
The category $\C_4$ has four simple objects, $V_0=\one=X_\varnothing$, $V_1=X_1$,
$V_2=X_2$ and $V_3=X_{12}$. The projective covers of $V_0$ and $V_1$ are the same
as the projective covers of $\one$ and $V$ in $\C_3$, while the
projective covers of $V_2$ and $V_3$ are the same as the projective
covers of $\one$ and $V$ in $\C_2$. Abbreviating $V_m$ to $m$, the
structures are as follows:
\[ P_0=\vcenter{\xymatrix@=3mm{&0\lkdl\lkdr \\ 1\lk[d] & & 0\lk[d] \\ 
0\lkdr & & 1\lkdl \\& 0}},\qquad
P_1=[1,0,0,1],\qquad P_2=[2,2],\qquad P_3=[3]. \]
Tensor products of simples are given by the following table:
\[ \renewcommand{\arraystretch}{1.4}
\begin{array}{|c|c|c|c|}
\hline
V_0 & V_1 & V_2 & V_3 \\ \hline
V_1 & [0,0] & [3] & P_2 \\ \hline
V_2 & [3] & [0,1,0] & P_1 \\ \hline
V_3 & P_2 & P_1 & P_0 \\ \hline
\end{array} \]
The Frobenius map is given by
\[ V_0^{(1)}=V_0\qquad V_{1}^{(1)}=V_2,\qquad V_2^{(1)}=0,\qquad V_{3}^{(1)} = 0, \] 
and the Frobenius--Perron dimensions of the simples are given by
\[ \FPdim V_0 = 1, \quad \FPdim V_{1} = \sqrt 2, \quad \FPdim V_{2} = \sqrt{2+\sqrt 2},
\quad \FPdim V_{3} = \sqrt{4+2\sqrt 2}. \]
The cohomology ring is the same as for $\C_3$.\bigskip

\subsubsection{The category $\C_5$} 
The category $\C_5$ also has four simple objects, $V_0$, $V_1$,
$V_2$ and $V_3$, with the same tensor products (but with the
projectives interpreted in $\C_4$), the same
Frobenius--Perron dimensions and the same Frobenius map, but the
projective covers are as follows.\par
{\tiny
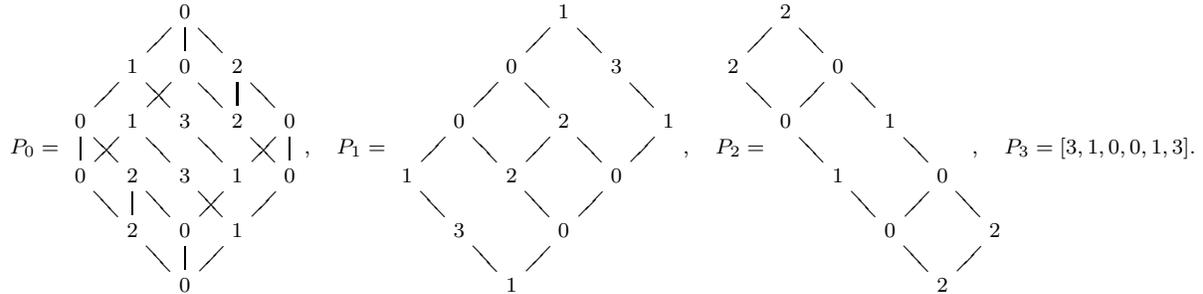
\begin{gather*} 
P_0=\vcenter{\xymatrix@=3mm{&&0\lkdlr\lk[d] \\
&1\lkdlr &0\lkdlr &2\lkdr\lk[d] \\
0\lkdr\lk[d] &1\lkdlr &3\lkdr &2\lkdr &0\lkdl\lk[d] \\
0\lkdr &2\lkdr\lk[d] &3\lkdr &1\lkdl &0\lkdl \\
&2\lkdr &0\lk[d] &1\lkdl \\
&&0}}, \quad
P_1=\vcenter{\xymatrix@=3mm{&&&1\lkdlr\\ 
&&0\lkdlr & &3\lkdr \\ 
&0\lkdlr &&2\lkdlr&&1\lkdl\\ 
1\lkdr &&2\lkdr &&0\lkdl\\
&3\lkdr&&0\lkdl\\ &&1}},\quad
P_2=\hspace{-7mm}
\vcenter{\xymatrix@=3mm{&2\lkdlr \\ 
2\lkdr &&0\lkdlr \\ 
&0\lkdr &&1\lkdr \\
&&1\lkdr &&0\lkdlr \\
&&&0\lkdr && 2\lkdl \\ 
&&&&2}}\hspace{-5mm},\quad
P_3=[3,1,0,0,1,3].
\end{gather*}
}

The Poincar\'e series for the cohomology ring appears to be as follows.\par
\begin{align*}
\sum t^i\dim\Ext^i(\one,\one) &= \frac{1+t^2+t^3+2t^4+t^5+t^6+t^8}{(1-t)(1-t^3)(1-t^7)} \\
&= 1+t+2t^2+4t^3+6t^4+8t^5+11t^6+14t^7+\cdots
\end{align*}
(we have computed it up to degree $40$). 

The Cartan matrix, the dimension of $\Ext^1_{\C_5}$, and of
$\Ext^2_{\C_5}$ between simples is as follows.\par
{\tiny
\[ \renewcommand{\arraystretch}{1.4}
\begin{array}{|c|cccc|}
\hline
&V_0&V_1&V_2&V_3 \\ \hline
V_0&8&4&4&2 \\
V_1&4&4&2&2 \\
V_2&4&2&4&0 \\
V_3&2&2&0&2 \\ 
\cline{1-5}
\multicolumn{5}{c}{\text{Cartan}}
\end{array} \qquad
\begin{array}{|c|cccc|}
\hline
&V_0&V_1&V_2&V_3 \\ \hline
V_0&1&1&1&0 \\
V_1&1&0&0&1 \\
V_2&1&0&1&0 \\
V_{3}&0&1&0&0 \\ \cline{1-5}
\multicolumn{5}{c}{\Ext^1_{\C_5}}
\end{array}\qquad
\begin{array}{|c|cccc|}
\hline
&V_0&V_1&V_2&V_3 \\ \hline
V_0&2&1&1&0 \\
V_1&1&1&0&0 \\
V_2&1&0&2&1 \\
V_3&0&0&1&1 \\ \cline{1-5}
\multicolumn{5}{c}{\Ext^2_{\C_5}}
\end{array}
 \]
}

\subsubsection{The category $\C_6$} 

The category $\C_6$ has eight simple objects, $V_0$ to $V_7$.
The projective covers of $V_0$ to $V_3$ are as in
$\C_5$, while the projective covers of $V_4$ to $V_7$ are as in
$\C_4$ but with all subscripts increased by four. Tensor products
are given by the following table.\par
{\tiny
\renewcommand{\arraystretch}{1.5}
\[ \begin{array}{|c|c|c|c|c|c|c|c|}
\hline
V_0&V_1&V_2&V_3&V_4&V_5&V_6&V_7\\ \hline
V_1&[0,0]\\ \cline{1-3}
V_2&[3]&[0,1,0]\\ \cline{1-4}
V_3&[2,2]&[1,0,0,1]
 & W
\\ \cline{1-5}
V_4&[5]&[6]&[7]&[0,2,0]\\ \cline{1-6}
V_5&[4,4] &[7]&P_6&[1,3,1] & W'
\\ \cline{1-7}
V_6&[7]&[4,5,4]
 &P_5&[2,0,1,0,2]
 & P_3
 &W''
\\ \hline
V_7&P_6&P_5&P_4&P_3&P_2&P_1&P_0 \\ \hline
\end{array} \]}%
where 
{\tiny
\[ W=\vcenter{\xymatrix@=1.2mm{&&0\lkdlr \\ &1\lkdl & &0\lkdl \\ 
 0\lkdr & &1\lkdl \\& 0}},\quad
W'=\vcenter{\xymatrix@=1.2mm{&0\lkdlr \\ 0\lkdr&&2\lkdlr\\
  &2\lkdr&&0\lkdl\\ &&0}},\qquad
W''=\vcenter{\xymatrix@=1.2mm{&&0\lkdlr\\&1\lkdlr&&2\lkdr\\0\lkdr&&3\lkdr&&0\lkdl\\
  &2\lkdr&&1\lkdl\\&&0}}. \]
}

The Frobenius--Perron dimensions and the Frobenius twists of the simple
modules are as follows.\par
{\tiny \newcommand{\pp}{\!+\!}
\[ \renewcommand{\arraystretch}{2}
\setlength{\arraycolsep}{0.7mm}
\begin{array}{|c||c|c|c|c|c|c|c|c|}
\hline
V &V_0&V_1&V_2&V_3&V_4&V_5&V_6&V_7 \\ \hline
\FPdim V & 1&\sqrt 2 & \sqrt{2\pp\sqrt{2}} & \sqrt{4\pp 2\sqrt{2}} &
 \sqrt{2\pp\sqrt{2+\sqrt{2}}}& \sqrt{4\pp 2\sqrt{2\pp\sqrt{2}}}&
 \sqrt{2\pp\sqrt{2}}.\sqrt{2\pp\sqrt{2\pp\sqrt{2}}}&  
 \sqrt{4\pp 2\sqrt{2}}.\sqrt{2\pp\sqrt{2\pp\sqrt{2}}}
\\ \hline
V^{(1)} & V_0 & V_2 & V_4 & V_6 & 0 & 0 & 0 & 0 \\
\hline \end{array} \]
}

\subsubsection{The category $\C_7$}
The Poincar\'e series for the cohomology ring appears to be as follows.\par
\begin{align*}
\sum t^i\dim\Ext^i(\one,\one) &= {\scriptstyle \frac{1 + t^2 + t^3 + 3 t^4 + 4 t^5 + 4 t^6 + 3 t^7 + 5 t^8 + 4 t^9 + 4 t^{10} + 4 t^{11} 
+ 4 t^{12} + 4 t^{13} + 5 t^{14} + 3 t^{15} + 4 t^{16} + 4 t^{17} + 3 t^{18} + t^{19} 
+ t^{20} + t^{22}}{(1-t)(1-t^3)(1-t^7)(1-t^{15})} }
\end{align*}
(we computed it up to degree 26). 

\begin{remark} The cohomology computations were done as follows. 
First we computed the tilting modules $T_m$ for $m$ large enough.
Then we computed the basic algebra of $\C_{2n}$ as $\End(\coplus_{i=2^n-1}^{2^{n+1}-2} T_i)$. 
Then we used a standard MAGMA function to compute the dimensions of 
${\rm Ext}^j$ between the simple modules over this basic algebra, for $j$ up to a specified point.
\end{remark}

\begin{remark} In general, on the basis of these examples, we expect the following properties of the cohomology. Let $R_n$ be the graded ring $\Ext^\bullet(\one,\one)$ in $\C_{2n-1}$ 
and $\C_{2n}$, and let $h_n(t)$ be its Hilbert series. Then we expect that:

1)  $R_n$ has a natural polynomial subalgebra $\k[x_1,\dots,x_n]$, where $|x_i|=2^i-1$, over which it is a free module of rank $2^{n(n-1)/2}$. In particular, $R_n$ is Cohen-Macaulay;

2) one has 
$$
h_n(t)=\frac{P_n(t)}{\prod_{i=1}^n(1-t^{2^i-1})},
$$
where $P_n(t)$ is a polynomial with nonnegative coefficients of degree $2^{n+1}-2(n+1)$; 

3) the polynomial $P_n$ is palindromic;

4) $R_n$ is Gorenstein;  

5) the homomorphism $R_n\to R_{n+1}$ induced by the inclusion $\C_{2n}\to \C_{2n+2}$ is an injection; 

6) the direct limit $\lim_{n\to\infty}R_n$ has finite dimensional homogeneous subspaces; 
hence there exists a coefficientwise limit $h_\infty(t):=\lim_{n\to \infty}h_n(t)$, which is the Hilbert series 
of $\Ext^\bullet(\one,\one)$ in $\C_\infty$. 

It is an interesting question whether $R_n$ is an integral domain. 
If so, then by Stanley's criterion (\cite{St}, Theorem 4.4), (1) and (3) imply (4). 
\end{remark}


\begin{thebibliography}{9999999}

\bibitem{Alperin:1980b}
J.~L. Alperin, {{Diagrams for modules}}, J.~Pure \& Applied Algebra
  \textbf{16} (1980), 111--119.
\bibitem{Benson/Carlson:1987a}
D.~J. Benson and J.~F. Carlson, {{Diagrammatic methods for modular
  representations and cohomology}}, Commun.\ in Algebra \textbf{15} (1987),
  53--121.
 \bibitem{BEO} D. Benson, P. Etingof, and V. Ostrik, New incompressible symmetric tensor categories in positive characteristic, arXiv:2003.10499.
 \bibitem{BCP} W. Bosma, J. Cannon and C. Playoust,
The Magma algebra system, I. The user language,
J.\ Symbolic Comput., v. 24, p. 235--265, 1997.
\bibitem{Cartier:1956a}
P.~Cartier, \emph{{Dualit\'e de Tannaka des groupes et alg\`ebres de Lie}},
  Comptes Rendus Acad.\ Sci.\ Paris, S\'erie I \textbf{242} (1956), 322--325.  
\bibitem{CEH} K. Coulembier, I. Entova-Aizenbud, T. Heidersdorf, Monoidal abelian envelopes and a conjecture of Benson-Etingof, arXiv:1911.04303.
\bibitem{DN1} A. Davydov and D. Nikshych, 
The Picard crossed module of a braided tensor category, Algebra and Number Theory, 
    Volume 7, Number 6 (2013), 1365---1403.
\bibitem{DN2} A. Davydov and D. Nikshych, Braided extensions of braided fusion categories, (in preparation). 
\bibitem{DR} A. Davydov and I. Runkel, $\ZZ/2\ZZ$-extensions of Hopf algebra module categories by their base categories, Advances in Mathematics, Volume 247, pages 192---265.
\bibitem{De1} P. Deligne, Cat\'egories tannakiennes, The Grothendieck Festschrift, vol. II, pp 111-195, Birkh\"auser, 1990. 
\bibitem{De2} P. Deligne, Cat\'egories tensorielles.
Moscow Math. J. 2(2002), no. 2, 227--248.
\bibitem{De3}
P. Deligne, 
La Cat\'egorie des Repr\'esentations du Groupe Sym\'etrique 
$S_t$, lorsque $t$ n'est pas un Entier Naturel,
http://www.math.ias.edu/\~\ phares/deligne/preprints.html
\bibitem{Deligne/Milne:1982a}
P.~Deligne and J.~S. Milne, \emph{{Tannakian Categories}}, Hodge cycles,
  motives, and Shimura varieties (P.~Deligne, J.~S. Milne, A.~Ogus, and {K.-y.}
  Shih, eds.), Lecture Notes in Mathematics, vol. 900, Springer-Verlag,
  Ber\-lin/New York, 1982, pp.~101--228.
\bibitem{Doty/Henke:2005a}
S.~R. Doty and A.~Henke, {{Decomposition of tensor products of modular
  irreducibles for $\mathsf{SL}_2$}}, Quarterly J.\ Math (Oxford) \textbf{56}
  (2005), no.~2, 189--207.
\bibitem{EM} S. Eilenberg and S. Mac Lane, On the groups $H(\Pi,n)$ II, Methods of Computation, Annals of Mathematics, v.70(1), 1954, p.49---139. 
\bibitem{EHS} I. Entova-Aizenbud, V. Hinich, and V. Serganova, Deligne categories and the limit of categories ${\rm Rep}(GL(m|n))$, to appear in IMRN,
arXiv:1511.07699.
\bibitem{EGNO} P. Etingof, S. Gelaki, D. Nikshych, and V. Ostrik, 
Tensor categories, AMS, 2015. 
\bibitem{ENO} P. Etingof, D. Nikshych, and V. Ostrik, 
Fusion Categories and Homotopy Theory, Quantum Topology, 1.3, 2010, pp. 209--273.
\bibitem{Fe} W. Feit, The representation theory of finite groups. 
\bibitem{J} J. C. Jantzen, Representations of algebraic groups,
  AMS, 2003. 
\bibitem{Joyal/Street:1991a}
A.~Joyal and R.~Street, \emph{{An introduction to Tannaka duality and quantum
  groups}}, Category theory (A.~Carboni, M.~C. Pedicchio, and G.~Rosolini,
  eds.), Lecture Notes in Mathematics, vol. 1488, Springer-Verlag, Ber\-lin/New
  York, 1991, pp.~413--492.
\bibitem{Krein:1949a}
M.~Krein, \emph{{A principle of duality for a bicompact group and square block
  algebra}}, Dokl.\ Akad.\ Nauk.\ SSSR \textbf{69} (1949), 725--728.
\bibitem{Morris:1977a}
S.~A. Morris, \emph{{Pontryagin duality and the structure of locally compact
  abelian groups}}, London Math.\ Soc.\ Lecture Note Series, vol.~29, Cambridge
  University Press, 1977.
\bibitem{O1} V. Ostrik, Module Categories Over Representations of $SL_q(2)$ in the Non-Semisimple Case,
Geometric and Functional Analysis, Volume 17, Issue 6, pp. 2005--2017. 
\bibitem{O2} V. Ostrik, On   symmetric   fusion   categories   in   positive   characteristic,
arXiv:1503.01492.
\bibitem{O3} V. Ostrik, private communication, 2017. 
\bibitem{SaavedraRivano:1972a}
N.~Saavedra Rivano, \emph{{Cat\'egories Tannakiennes}}, Lecture Notes in
  Mathematics, vol. 265, Springer-Verlag, Ber\-lin/New York, 1972.
\bibitem{St} R. P. Stanley, Hilbert functions of graded algebras, Advances in Mathematics, v.28, p.57--83, 1978. 
\bibitem{Tannaka:1938a}
T.~Tannaka, \emph{{\"Uber den Dualit\"atssatz der nichtkommutativen
  topologischen Gruppen}}, Tohoku Math.\ J. \textbf{45} (1938), 1--12.
\bibitem{Ven} S. Venkatesh, Hilbert Basis Theorem and Finite Generation of Invariants in Symmetric Fusion Categories in Positive Characteristic, International Mathematics Research Notices 2016(16),
DOI10.1093/imrn/rnv305. 
\bibitem{W} L. Washington, Introduction to Cyclotomic Fields, Graduate texts in Mathematics, Second edition, Springer, 1996.
\end{thebibliography}
\end{document}